\documentclass[11pt]{article}
\usepackage[a4paper, total={17cm,25cm}]{geometry}

\usepackage{amsmath,amssymb,amsfonts}
\usepackage{tikz,tkz-euclide}
\usetikzlibrary{math,angles,arrows.meta,patterns,calc}
\usepackage{enumitem}
\usepackage{color, graphics}
\usepackage{hyperref}

\usepackage[alphabetic,y2k,initials]{amsrefs}

\usepackage[normalem]{ulem}

\numberwithin{equation}{section}

\newtheorem{theorem}{Theorem}[section]
\newtheorem{proposition}[theorem]{Proposition}
\newtheorem{corollary}[theorem]{Corollary}
\newtheorem{lemma}[theorem]{Lemma}
\newtheorem{remark}[theorem]{Remark}

\newtheorem{definition}[theorem]{Definition}

\newenvironment{proof}[1][Proof]{\noindent\textit{#1.} }{\hfill$\Box$\medskip}

\title{Is every triangle a trajectory of an elliptical billiard?}

\usepackage{authblk}
\author[1,3]{Vladimir Dragovi\'c}
\author[2,3]{Milena Radnovi\'c}
\affil[1]{\textsc{The University of Texas at Dallas, Department of Mathematical Sciences}}
\affil[2]{\textsc{The University of Sydney, School of Mathematics and Statistics}}
\affil[3]{\textsc{Mathematical Institute SANU, Belgrade}}
\affil[ ]{\texttt{vladimir.dragovic@utdallas.edu, milena.radnovic@sydney.edu.au}}

\date{}

\begin{document}
	
	\maketitle

\begin{abstract} Using Marden's Theorem from geometric theory of polynomials, we show that for every triangle there is a unique ellipse such that the triangle is a billiard trajectory within that ellipse.
Since $3$-periodic trajectories of billiards within ellipses are examples of the Poncelet polygons, our considerations provide a new insight into the relationship between Marden's Theorem and the Poncelet Porism, two gems of exceptional classical beauty.  We also show that every parallelogram is a billiard trajectory within a unique ellipse. We prove a similar result for the self-intersecting polygonal lines consisting of two pairs of congruent sides, named ``Darboux butterflies". In each of three considered cases, we effectively calculate the foci of the boundary ellipses.
\end{abstract}


\emph{AMS subclass:} 51N20, 30C15, 37J35, 70H06
%\newpage

\tableofcontents

\

\

\section{Introduction}\label{sec:intro}
Recall that \emph{mathematical billiard} in a planar domain is a dynamical system where a particle moves without constraints within the domain, and obeys \emph{the billiard reflection law} when it hits the boundary \cites{KozTrBIL, Tab2005book}.
Thus, billiard trajectories are polygonal lines with vertices at the boundary, such that two consecutive sides form congruent angles with the tangent line to the boundary at their joint vertex, see Figure \ref{fig:billiard-reflection}.
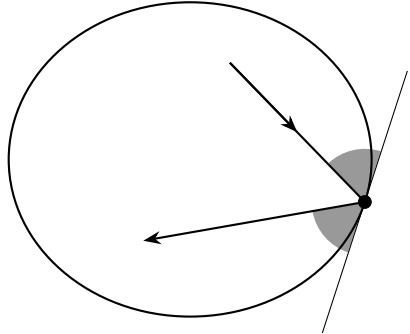
\begin{figure}[h]
	\begin{center}
\begin{tikzpicture}[scale=1.2]
	\tikzmath{\koef=3; \k1=0.7; \k2=1-\k1; \k3=0.35; \k4=1-\k3;}
	\coordinate (A) at (1.92538, -0.468715);
	\coordinate (T2) at (1.92538+\koef*0.468715/3,-0.468715+\koef*1.92538/4);
	\coordinate (T1) at (1.92538-\koef*0.468715/3,-0.468715-\koef*1.92538/4);
	\coordinate (B) at (\k2*1.92538-0.199667*\k1,-0.468715*\k2+\k1* 1.7234);
	\coordinate (C) at (\k2*1.92538-1.56424*\k1,-0.468715*\k2-\k1*1.07929);
	\coordinate (B1) at (\k4*1.92538-0.199667*\k3,-0.468715*\k4+\k3* 1.7234);

	\draw[white] (T2) -- (A) -- (B)
	pic [fill=black!40, angle radius=20pt] {angle = T2--A--B};
	
	\draw[white] (T1) -- (A) --  (C)
	pic [fill=black!40, angle radius=20pt] {angle = C--A--T1};
	
	\draw[thick](B)--(A);
	\draw[thick,-Stealth](B)--(B1);
	
	\draw[thick,-Stealth](A)--(C);
	\draw[thick] (0,0) ellipse (2 and 1.732);
	\draw[black, fill=black] 	(A) circle (2pt);
	
	\draw (T1) -- (T2);
	
\end{tikzpicture}
\caption{Billiard reflection law: the angle of incidence with the tangent line at the bouncing point on the boundary equals the angle of reflection.}\label{fig:billiard-reflection}
\end{center}
\end{figure}
Mathematical billiards are idealized models in many aspects: a usual billiard ball is replaced by a material point, the friction and spin are neglected.
Such models have natural applications, for example in geometric optics.
%Billiard dynamics has two different regimes: the first one is inside the billiard domain, and the second concerns with the impacts off the boundary.
%We assume here that the impacts are \emph{absolutely elastic}, which is equivalent to the geometric billiard law: ~the impact and reflection angles are congruent to each other and the speed remains unchanged after the impacts.
Here, we assume that the billiard particle is of the unit mass and it moves under the inertia between the impacts, i.e.~uniformly along straight lines.

Billiards within ellipses have been intensively studied, see for example \cites{Bolotin1990,Kozlov2003,  KozTrBIL, DragRadn2014jmd, ADSK2016, KS2018, BiMi2017, BiMi2022, CZ2021, GKR2021, GKR2022, DGR2022, FoVed} and references therein.
Their generalisations to higher-dimensions \cites{DragRadn2011book,St,DR2019cmp,DR2023adv} and various geometric settings
\cites{Glu2021,DR2023} have also been in the focus of research interest.
A common denominator in the wide variety of approaches, methods and questions in all those works is rich interplay with geometry of conics and quadrics, see for example \cites{Tab2022,GSO}.
This work offers a novel glimpse into that interplay.

In this note, we will give the affirmative answer to the title question. This answer may appear surprising at the first glance, if we recall that a conic is defined with various sets of five conditions, while a triangle as an inscribed billiard trajectory imposes six conditions, three points and three tangent lines. Regardless of these seemingly overdetermined conditions, the solution conic exists and it is unique always, and it is always an ellipse.
The main ingredients in our considerations of triangular trajectories  are  Marden's theorem from the geometry of polynomials \cite{MardenPOLY} and the classical Ceva's theorem from elementary geometry.
Periodic trajectories of billiards within ellipses can be seen as an instance of so-called Poncelet polygons, which are closed polygonal lines inscribed in one conic an circumscribed about the other conic, see e.g. \cite{DragRadn2011book}.
The Poncelet Theorem from projective geometry of conics states that if for a  given pair of conics there is one Poncelet polygon, then there are infinitely many such polygons and all have the same number of sides.
Thus, we are in this work dealing with an interaction between Marden's theorem and the Poncelet theorem, which are both recognized by their exceptional classical beauty.
Previously, a strong relationship between these two theorems was observed in \cite{Drag2011} in a different context.

This paper is organised as follows.
We review basic facts about billiards within triangles and the theorems of Ceva, Menelaus, and Simson in Section \ref{sec:triangles}; 
conics and elliptical billiards in Section \ref{sec:conics}; and Marden's theorem in Section \ref{sec:marden}.
In Section \ref{sec:3periodic}, we prove that each triangle is a billiard trajectory within a unique ellipse.
In Section \ref{sec:convex4}, we give a complete characterization of convex $4$-periodic elliptical billiard trajectories as parallelograms and show the converse statement: that each parallelogram is a billiard trajectory within unique ellipse, see Theorem \ref{th:parall}.
Nonconvex $4$-periodic elliptical billiard trajectories are considered in Section \ref{sec:Darboux}, where they are characterized as so-called \emph{Darboux butterflies}.
We also show that every Darboux butterfly is a billiard trajectory within a unique ellipse, see Theorem \ref{th:each-butterfly}. 
In each of the three cases, of a triangle, a parallelogram, and a Darboux butterfly, we effectively calculate the foci of the boundary ellipse.
We provide multiple proofs for our statements, allowing interplay of various classical and modern geometry results.

\section{Triangles}\label{sec:triangles}

In this section consists of two parts: in Subsection \ref{sec:triangular-billiards}, we review main properties of the billiards within a triangle in the Euclidean plane, and in Subsection \ref{sec:ceva} we revisit the Ceva theorem.

\subsection{Triangular billiards}
\label{sec:triangular-billiards}

A trajectory of such a billiard is a polygonal line, finite or infinite, with vertices on the sides of the
triangle, such that consecutive edges of the trajectory satisfy the
billiard reflection law. i.e.~they form the same angle with the side of the
triangle which their common vertex lie on, see Figure \ref{fig:billiard-triangle}.
	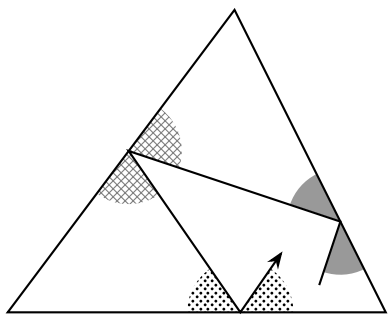
\begin{figure}[h]
	\begin{center}
		\begin{tikzpicture}
			\coordinate (A) at (-3,0);
			\coordinate (B) at (2,0);
			\coordinate (C) at (0,4);
			\coordinate (M1) at (1.12,0.36);
			\coordinate (M2) at (1.4,1.2);
			\coordinate (M3) at (-1.4,2.13333);
			\coordinate (M4) at (0.0769231,0);
			\coordinate (M5) at (0.635236,0.806452);
			
			\draw[white] (M1) -- (M2) -- (B)
			pic [fill=black!40, angle radius=20pt] {angle = M1--M2--B};
			\draw[white] (C) -- (M2) -- (M3)
			pic [fill=black!40, angle radius=20pt] {angle = C--M2--M3};
			
			\draw[white] (C) -- (M3) -- (M2)
			pic [pattern=crosshatch, pattern color=gray, angle radius=20pt] {angle = M2--M3--C};
			\draw[white] (A) -- (M3) -- (M4)
			pic [pattern=crosshatch, pattern color=gray, angle radius=20pt] {angle = A--M3--M4};
			
			\draw[white] (A) -- (M4) -- (M3)
			pic [pattern=crosshatch dots, pattern color=black, angle radius=20pt] {angle = M3--M4--A};
			\draw[white] (B) -- (M4) -- (M5)
			pic [pattern=crosshatch dots, pattern color=black, angle radius=20pt] {angle = B--M4--M5};

			\draw[thick] (A)--(B)--(C)--cycle;
						
			\draw[thick,-Stealth]
			(M1)--(M2)--(M3)--(M4)--(M5);
		\end{tikzpicture}
		\caption{Billiard motion in a triangle.}\label{fig:billiard-triangle}
	\end{center}
\end{figure}

The reflection is not well defined in the vertices of the triangle, thus we omit from our consideration trajectories falling in a vertex.

\begin{remark}\label{rem:min} Note the following minimization property of the billiard reflection: if $X$, $Y$ are two points on the same side of the a given line $\ell$, then the length $XL+LY$, for $L\in\ell$, will be the smallest when the segments $XL$ and $LY$ satisfy the billiard reflection law off $\ell$, see Figure \ref{fig:minimization}.
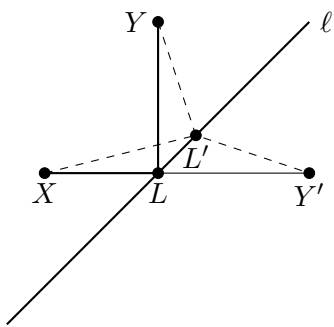
\begin{figure}[h]
	\begin{center}
		\begin{tikzpicture}[scale=1]
			\draw[thick](-2,-2)--(2,2) node[right]{$\ell$};
			
			\draw[thick] (-1.5,0)--(0,0)--(0,2);
			\draw(0,0)--(2,0);
			\draw[dashed](0,2)--(0.5,0.5)--(2,0);
			\draw[dashed](-1.5,0)--(0.5,0.5);
			\draw[black,fill=black](-1.5,0)circle(2pt)node[below]{$X$};
			\draw[black,fill=black](0,2)circle(2pt)node[left]{$Y$};
			\draw[black,fill=black](0,0)circle(2pt)node[below]{$L$};
			\draw[black,fill=black](2,0)circle(2pt)node[below]{$Y'$};
			\draw[black,fill=black](0.5,0.5)circle(2pt)node[below]{$L'$};
		\end{tikzpicture}
		\caption{The minimization property of billiard reflection: the shortest path connecting points $X$ and $Y$ which visits line $\ell$ is the billiard path $XLY$. Its length equals the segment $XY'$, where $Y'$ is symmetric to $Y$ with respect to $\ell$. For any other point $L'\in\ell$, we have $XL'+L'Y=XL'+L'Y'>XY'$.}\label{fig:minimization}
	\end{center}
\end{figure}
\end{remark}

One of the first natural questions for any dynamical system is to establish if its periodic trajectories exist and find them if they do.
Thus, next we recall a proof of the existence a periodic trajectory within any acute triangle.

\begin{theorem}[see e.g.~\cite{DragRadn2011book}*{Theorem 2.2}]\label{th:acute}
	Let $\triangle ABC$ be an acute triangle, and $K$, $L$, $M$ the feet of its altitudes.
	Then $\triangle KLM$ is the triangle with minimal perimeter inscribed in $\triangle ABC$ and it represents a unique $3$-periodic trajectory of the billiard with $\triangle ABC$.
\end{theorem}

\begin{proof}
Let $M'$ be an arbitrary given point on side $AB$. In order to find points
	$K'\in BC$, $L'\in AC$ such that the triangle $K'L'M'$ has the minimal
	perimeter, denote by $M_1$, $M_2$ the points symmetric to $M'$ with
	respect to sides $BC$ and $AC$ respectively. Then $K'$ and $L'$ are
	intersection points of $M_1M_2$ with sides $BC$ and $AC$ respectively, see
	Figure \ref{fig:M1M2}.
	
	\begin{figure}[h]
		\centering
			\begin{tikzpicture}
				\filldraw[thick]
				(-3,0) circle (2pt) node[below left] {$A$} --
				(2,0) circle (2pt) node[below right] {$B$} --
				(0,3.5) circle (2pt) node[above right] {$C$}
				-- (-3,0);
				
				\draw[gray]
				 (-1,0)  -- (-3.30588, 1.97647)-- (3.52308, 2.58462) -- cycle;
				
				\draw[gray,fill=gray] (-3.30588, 1.97647) circle (2pt) node[left, black]{$M_2$};
				\draw[gray,fill=gray] (3.52308, 2.58462) circle (2pt) node[right, black]{$M_1$};
				
				\draw[thick, dashed]
				(-1.1406, 2.1693)  -- (0.668348, 2.33039)  -- (-1,0) -- cycle;
				
				\draw[black, fill=black]
				(-1.1406, 2.1693) circle (2pt) node[above left] {$L'$};
				\draw[black, fill=black]
				(0.668348, 2.33039) circle (2pt) node[above right] {$K'$};
				\draw[black, fill=black]
				(-1,0) circle (2pt) node[below right] {$M'$};

			\end{tikzpicture}
			\caption{For a fixed point $M'$ on $AB$, we construct points $K'$, $L'$ on the remaining two sides of the triangle, such that $\triangle K'L'M'$ has smallest  possible perimeter.}\label{fig:M1M2}
	\end{figure}
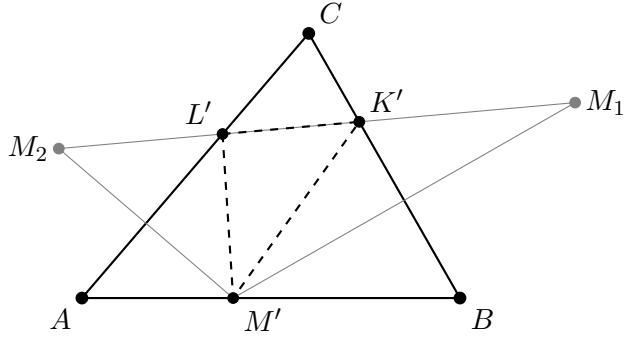
	
	The perimeter of $\triangle K'L'M'$ is equal to the length of segment $M_1M_2$.
	We observe that $M_1M_2$ is a side of the isosceles triangle $CM_1M_2$, whose angle $\angle M_1CM_2=2\angle BCA$ does not depend on the choice of point $M'$.
	Since $CM_1\cong CM_2\cong CM'$, the segment $M_1M_2$ is shortest when the point $M'$ is such that $CM'$ is the altitude
	of the triangle $ABC$ from the vertex $A$, i.e.~the minimal perimeter is attained when $M'=M$.
	It immediately follows that points $K'$ and $L'$, as constructed above, are then also the feet of the corresponding altitudes, see Figure
	\ref{fig:KLM}.
	
	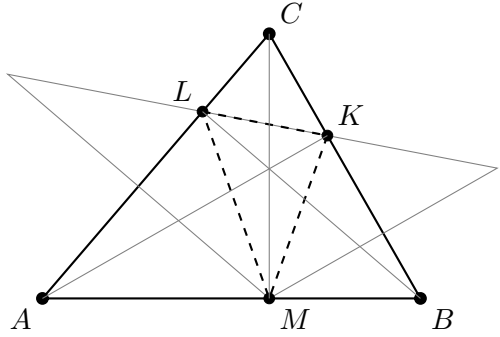
\begin{figure}[h]
		\begin{center}
			\begin{tikzpicture}
				\filldraw[thick]
				(-3,0) circle (2pt) node[below left] {$A$} --
				(2,0) circle (2pt) node[below right] {$B$} --
				(0,3.5) circle (2pt) node[above right] {$C$}
				-- (-3,0);
				
				\draw[gray]
				(0,0)  -- (-3.45882, 2.96471)--(3.01538, 1.72308) -- cycle;

				\draw[thick, dashed]
				(-0.882353, 2.47059)  -- (0.769231, 2.15385)  -- (0,0) -- cycle;
				
				\draw[black, fill=black]
				(-0.882353, 2.47059) circle (2pt) node[above left] {$L$};
				\draw[black, fill=black]
				(0.769231, 2.15385) circle (2pt) node[above right] {$K$};
				\draw[black, fill=black]
				(0,0) circle (2pt) node[below right] {$M$};
				
				\draw[gray]	(0,0)  -- (0, 3.5);
				\draw[gray]	(0.769231, 2.15385)  -- (-3, 0);
				\draw[gray]	(-0.882353, 2.47059)  -- (2, 0);
				
			\end{tikzpicture}
			\caption{The vertices of the triangle with smallest perimeter inscribed in $\triangle ABC$ are the feet $K$, $L$, $M$ of the altitudes.}\label{fig:KLM}
		\end{center}
	\end{figure}
	
	The minimizing property from Remark \ref{rem:min} implies that $\triangle KLM$, as the triangle with smallest perimeter inscribed in $\triangle ABC$ is a billiard trajectory within $\triangle ABC$.

Now, let us show uniqueness.
Suppose that $K_1L_1M_1$ is another periodic trajectory within the triangle $ABC$, such that $K_1\in BC$, $L_1\in AC$, $M_1\in AB$.
Then, according to the billiard reflection law, we can denote: $k:=\angle CK_1L_1=\angle BK_1M_1$,
$l:=\angle CL_1K_1=\angle AL_1M_1$,
$m:=\angle AM_1L_1=\angle BM_1K_1$.
Denoting by $\angle A$, $\angle B$, $\angle C$ the angles of the triangle $ABC$, we get that the sums of the angles in the triangles $AL_1M_1$, $BK_1M_1$, $CK_1L_1$ are:
$$
\angle A+l+m=\angle B+k+m=\angle C+k+l=180^{\circ},
$$
which gives: $k=\angle A$, $l=\angle B$, $m=\angle C$.

From there, the sides of the triangle $K_1L_1M_1$ are parallel to the corresponding sides of the triangle $KLM$.
The assumption that $K_1L_1M_1$ is a periodic billiard trajectory then implies $K_1=K$, $L_1=L$, $M_1=M$, as shown in Figure \ref{fig:6periodic}.  Namely, if $K_1$ is between $K$ and $C$, then $M_1$ lies in the interior of $\triangle ABC$. This follows from the similarity of triangles $KLM$ and $K_1L_1M_1$, by using the Thales theorem. The same argument shows that $M_1$ lies outside $\triangle ABC$ if $K_1$ is between $K$ and $B$.
	\begin{figure}[h]
	\begin{center}
		\begin{tikzpicture}
			\filldraw[thick]
			(-3,0) circle (2pt) node[below left] {$A$} --
			(2,0) circle (2pt) node[below right] {$B$} --
			(0,3.5) circle (2pt) node[above right] {$C$}
			-- (-3,0);
			
			\draw[thick, dashed]
			(-0.882353, 2.47059)  -- (0.769231, 2.15385)  -- (0,0) -- cycle;
			
			\draw[black, fill=black]
			(-0.882353, 2.47059) circle (2pt) node[above left] {$L$};
			\draw[black, fill=black]
			(0.769231, 2.15385) circle (2pt) node[above right] {$K$};
			\draw[black, fill=black]
			(0,0) circle (2pt) node[below right] {$M$};
			
			\draw[thick, gray]
			(1, 1.75)  -- (-1.14706, 2.16176)  -- (-0.375003,0) -- (0.538459, 2.5577) -- (-0.617645, 2.77941) -- (0.375003, 0) -- cycle ;

		\end{tikzpicture}
		\caption{The trajectories with segments parallel to the segments of $KLM$ are $6$-periodic.}\label{fig:6periodic}
	\end{center}
\end{figure}
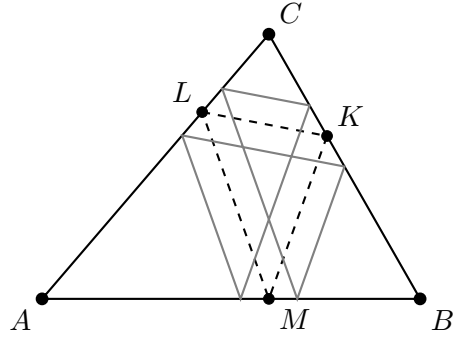
\end{proof}

	The previous proof  essentially relies on the assumption that $\triangle ABC$ is acute, since otherwise some of the feet of the altitudes are not inner points of the sides.
	Thus, additional discussion in needed for the right and obtuse triangles.
	
It is easy to see that there are periodic billiard trajectories within a right triangle: one of them, the polygonal line $KLMNMLK$, is shown in Figure
\ref{fig:right}.
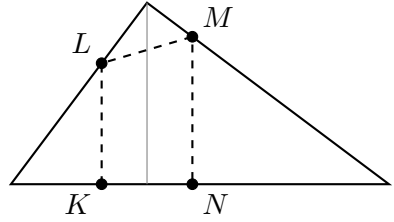
\begin{figure}[h]
	\begin{center}
		\begin{tikzpicture}[scale=1]
			\draw[thick](0,0)--(5,0)--(1.8,2.4)--cycle;
			
			\draw[gray](1.8,0)--(1.8,2.4);
			
			\draw[thick, dashed]
			(1.2,0)--(1.2, 1.6)--(2.4,1.95)--(2.4,0);
			
			\draw[black, fill=black]
			(1.2,0) circle (2pt) node[below left] {$K$};
			\draw[black, fill=black]
			(1.2, 1.6) circle (2pt) node[above left] {$L$};
			\draw[black, fill=black]
			(2.4,1.95) circle (2pt) node[above right] {$M$};
			\draw[black, fill=black]
			(2.4,0) circle (2pt) node[below right] {$N$};			
		\end{tikzpicture}
		\caption{A periodic trajectory of the billiard within a right triangle.}\label{fig:right}
	\end{center}
\end{figure}

After such elementary considerations for acute and right triangles, one can stay amazed by the fact that it is not known if billiards within general obtuse triangles have any periodic trajectories!
There are examples for some special cases and also an intriguing computer-based proof for the existence of periodic billiard trajectories when the obtuse angle does not exceed $100^{\circ}$, see \cites{Sch1, Sch2}.

\subsection{Theorems of Ceva and Menelaus}
\label{sec:ceva}

In this section, we will formulate the two classical theorems related to the geometry of triangles.

\begin{theorem}[Ceva's Theorem]\label{thm:ceva}
Let $ABC$ be a given triangle and $K$, $L$, $M$ points on the lines $BC$, $AC$, $AB$ respectively, such that none of them coincides with a vertex of the triangle.
Then the lines $AK$, $BL$, $CM$ are either concurrent or parallel if and only if:
	\begin{equation}\label{eq:ceva}
		{\frac {AM}{MB}}\cdot {\frac {BK}{KC}}\cdot {\frac {CL}{LA}}=1.
	\end{equation}
\end{theorem}
Ceva's theorem is illustrated in Figure \ref{fig:ceva}.
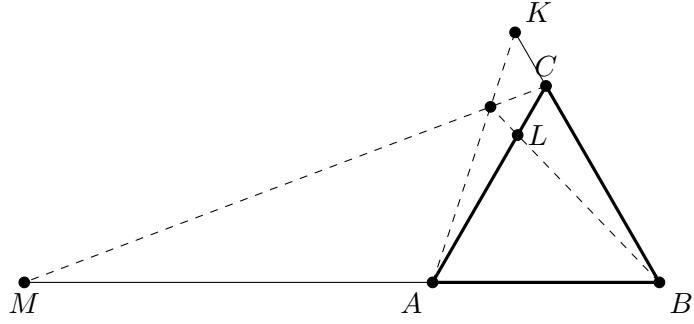
\begin{figure}[h]
	\begin{center}
		\begin{tikzpicture}
\coordinate (A) at (-1.5, 0);
\coordinate (B) at (1.5, 0);
\coordinate (C) at (0,1.5*1.73205);
\coordinate (K) at (-0.409091, 3.30664);
\coordinate (L) at (-0.375, 1.94856);
\coordinate (M) at (-6.9, 0);
\coordinate (O) at (-0.734043, 2.32169);

			\draw[very thick](A)--(B)--(C)--cycle;
			
\draw(A)--(M);
\draw(C)--(K);

\draw[dashed](B)--(O);
\draw[dashed](C)--(M);
\draw[dashed](A)--(K);
			
			\draw[black, fill=black]
			(A) circle (2pt) node[below left] {$A$};
			\draw[black, fill=black]
			(B) circle (2pt) node[below right] {$B$};
			\draw[black, fill=black]
			(C) circle (2pt) node[above] {$C$};
			\draw[black, fill=black]
			(K) circle (2pt) node[above right] {$K$};
			\draw[black, fill=black]
			(L) circle (2pt) node[right] {$L$};
			\draw[black, fill=black]
			(M) circle (2pt) node[below] {$M$};
			\draw[black, fill=black]
			(O) circle (2pt);
		\end{tikzpicture}
		\caption{Ceva's theorem.}\label{fig:ceva}
	\end{center}
\end{figure}
\begin{remark}In Ceva's Theorem, the signed lengths of segments are used in \eqref{eq:ceva}.
For example, in Figure \ref{fig:ceva}, the quantities $AM/MB$ and $BK/KC$ are negative, while $CL/LA$ is negative.
\end{remark}

\begin{remark}Notice that Ceva's theorem can be applied to an acute triangle $ABC$ and its $3$-periodic billiard trajectory $KLM$, since the altitudes intersect at one point, see Figure \ref{fig:KLM} and Theorem \ref{th:acute}.
\end{remark}

We will use  Ceva's theorem to prove the following statement.

\begin{lemma}\label{lem:ms} Let $ABC$ be a given triangle and $K$, $L$, $M$ its feet of the altitudes from $A$, $B$, $C$ respectively.
	Then there exist real numbers $m_1$, $m_2$, $m_3$ such that point $K$ divides $BC$ in the ratio $m_2:m_3$, point $L$ divides $CA$ in the ratio $m_3:m_1$, and $M$ divides $AB$ in the ratio $m_1:m_2$.
	Moreover
	\begin{equation}\label{eq:ms}
		m_1:m_2:m_3=(b^2+c^2-a^2):(a^2+c^2-b^2):(a^2+b^2-c^2),
	\end{equation}
	where $a$, $b$, $c$ respectively denote the lengths of the sides $BC$, $AB$, $AC$.
\end{lemma}
\begin{proof}
	Since the altitudes of a triangle intersect at the orthocenter, Ceva's theorem implies the existence of numbers $m_1$, $m_2$, $m_3$.
	The similarities of right triangles $\triangle KBA\sim \triangle MBC$, $\triangle LCB\sim \triangle KCA$, $\triangle MAC\sim \triangle LAB$ imply the following relations: $KB/AB=MB/BC$, $CL/CB=CK/AC$, $AM/AC=LA/AB$, yielding
	$$
	m_1:m_2:m_3=(AL\cdot AC):(BM\cdot BA):(CK\cdot CB).
	$$
	The Pythagorean  theorem gives
	$$
	CK^2=b^2-AK^2,\quad AL^2=c^2-BL^2,\quad BM^2=a^2-CM^2,
	$$
	thus
	$$
	m_1:m_2:m_3=b\sqrt{c^2-BL^2}: c \sqrt{a^2-CM^2}:a\sqrt{b^2-AK^2}.
	$$
	Observing that $a\cdot AK=b\cdot BL= c\cdot CM=2\mathcal A$, where $\mathcal{A}$ is the area of the triangle, we get:
	$$
	m_1:m_2:m_3
	=
	\sqrt{c^2b^2-4{\mathcal A}^2}:\sqrt{a^2c^2-4{\mathcal A}^2}:\sqrt{a^2b^2-4{\mathcal A}^2}.
	$$
Finally, applying the Heron formula $\mathcal{A}=\sqrt{s(s-a)(s-b)(s-c)}$, with $s=(a+b+c)/2$, and simplifying, we get \eqref{eq:ms}.
\end{proof}

\begin{theorem}[Theorem of Menelaus]\label{th:menelaus}
Let $\triangle ABC$ be a given triangle and $p$ a line intersecting the lines $BC$, $AC$, $AB$ in points $P$, $Q$, $R$ respectively, which do not coincide with any of the vertices of the triangle.
Then
\begin{equation}\label{eq:menelaus}
	{\frac {AR}{RB}}\cdot {\frac {BP}{PC}}\cdot {\frac {CQ}{QA}}=-1.
\end{equation}

 Moreover, the converse is also true: if points $P$, $Q$, $R$ are chosen on the lines $AB$, $BC$, $CA$, such that \eqref{eq:menelaus} is satisfied, then $P$, $Q$, $R$ are collinear.
\end{theorem}
The theorem of Menelaus is illustrated in Figure \ref{fig:menelaus}.
\begin{figure}[h]
	\begin{center}
		\begin{tikzpicture}
		\begin{scope}[rotate=-120]
			\coordinate (A) at (-1.5, 0);
			\coordinate (B) at (1.5, 0);
			\coordinate (C) at (0,1.5*1.73205);
			\coordinate (K) at (-0.409091, 3.30664);
			\coordinate (L) at (-0.375, 1.94856);
			\coordinate (M) at (-0.326087, 0);

			\draw[very thick](A)--(B)--(C)--cycle;
			
			\draw(C)--(K);
			
			\draw[very thick,gray,dashed](K)--(M);
			
				\draw[black, fill=black]
			(A) circle (2pt) node[above] {$A$};
			\draw[black, fill=black]
			(B) circle (2pt) node[below left] {$B$};
			\draw[black, fill=black]
			(C) circle (2pt) node[below] {$C$};
			
			\draw[black, fill=black]
			(K) circle (2pt) node[right] {$P$};
			\draw[black, fill=black]
			(L) circle (2pt) node[above right] {$Q$};
			\draw[black, fill=black]
			(M) circle (2pt) node[above left] {$R$};
		\end{scope}	
		\end{tikzpicture}
		\caption{Theorem of Menelaus.}\label{fig:menelaus}
	\end{center}
\end{figure}
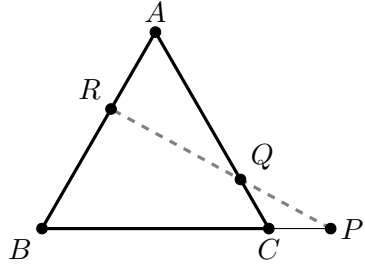

\begin{corollary}\label{cor:exterior-angles}
For a given triangle, consider the points where the bisectors of it exterior angles meet the opposite sides.
Then those three points are collinear.
\end{corollary}
\begin{proof}
Denote by $a$, $b$, $c$ respectively the bisectors of the exterior angles at the vertices $A$, $B$, $C$ of the triangle, and by $P$, $Q$, $R$ the intersections of those bisectors with the lines $BC$, $AC$, $AB$, as shown in Figure \ref{fig:exterior-angles}.
\begin{figure}[h]
	\begin{center}
		\begin{tikzpicture}
			\begin{scope}[rotate=120]
			\coordinate (A) at (-1.5, 0);
			\coordinate (B) at (1.5, 0);
			\coordinate (C) at (-0.75, 1.29904);
			\coordinate (P) at (-3, 2.59808);
			\coordinate (Q) at (4.09808, 9.69615);
			\coordinate (R) at (-5.59808, 0);

			\draw[very thick,fill=gray!30](A)--(B)--(C)--cycle;
			
			\draw[thin,gray](A)--(R);
			\draw[thin,gray](P)--(C)--(Q);
			
			\draw[very thick, gray](A)--(P);
			\draw[very thick, gray](B)--(Q);
			\draw[very thick, gray](C)--(R);
			
			\draw[very thick,dashed](Q)--(R);
			
			\draw[black, fill=black]
			(A) circle (2pt) node[right] {$A$};
			\draw[black, fill=black]
			(B) circle (2pt) node[above] {$B$};
			\draw[black, fill=black]
			(C) circle (2pt) node[above left]{$C$};
			
			\draw[black, fill=black]
			(P) circle (2pt) node[below left] {$P$};
			\draw[black, fill=black]
			(Q) circle (2pt) node[left] {$Q$};
			\draw[black, fill=black]
			(R) circle (2pt) node[below] {$R$};
			
			\node at (-2., 1.29904){$a$};
			\node at (-3.17404, 0.9){$c$};
			\node at (3, 4.84808){$b$};
		\end{scope}
		\end{tikzpicture}
		\caption{The bisectors of exterior angles of a triangle meet the opposite sides at collinear points.}\label{fig:exterior-angles}
	\end{center}
\end{figure}
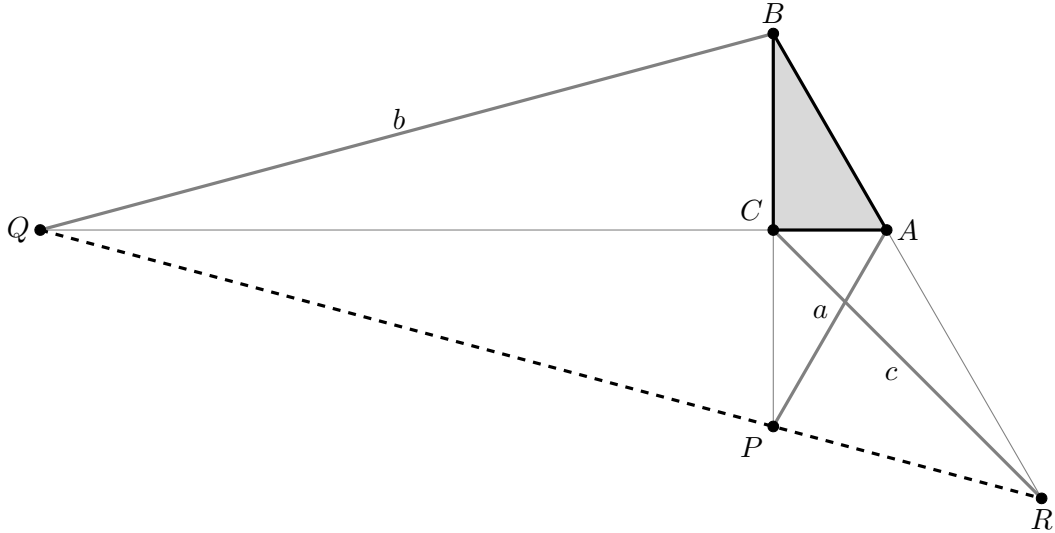

Since the intersection of the exterior bisector externally divides the opposite side in the ratio of two remaining sides of the triangle, the Theorem of Menelaus (Theorem \ref{th:menelaus}) gives that $P$, $Q$, $R$ are collinear.
\end{proof}

\begin{theorem}[Simson's line]\label{th:simson}
Let a triangle $ABC$ and a point $S$ be given in the plane.
Then the orthogonal projections of that point to the sides of the triangle are collinear if and only if point $S$ belongs to the circumcircle of the triangle.
(See Figure \ref{fig:simson}).
\end{theorem}
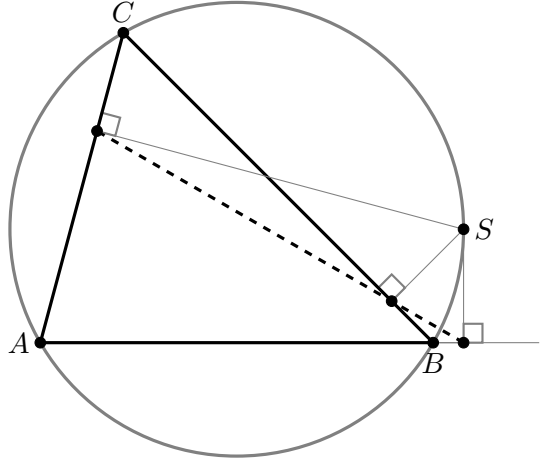
\begin{figure}[h]
	\begin{center}
		\begin{tikzpicture}
			\coordinate (A) at (-2.59808, -1.5);
			\coordinate (B) at (2.59808, -1.5);
			\coordinate (C) at (-1.5, 2.59808);
			\coordinate (P) at (3,0);
			\coordinate (P1) at (3., -1.5);
			\coordinate (P2) at (-1.84808, 1.29904);
			\coordinate (P3) at (2.04904, -0.950962);
			\coordinate (O) at (0,0);
			\coordinate(B1) at (4,-1.5);
			
				\tkzMarkRightAngle[thick,draw=gray](C,P2,P)
				\tkzMarkRightAngle[thick,draw=gray](C,P3,P)
				\tkzMarkRightAngle[thick,draw=gray](B1,P1,P)
			
			\draw[very thick](A)--(B)--(C)--cycle;
			
			\draw[very thick,dashed](P1)--(P2);
			
			\draw[thin,gray](P3)--(P)--(P2);
			\draw[thin,gray](P)--(P1)--(B)--(B1);
			\draw[gray,very thick] (O) circle (3);
			
			\draw[black, fill=black]
			(A) circle (2pt) node[left]{$A$};
			\draw[black, fill=black]
			(B) circle (2pt)node[below]{$B$};
			\draw[black, fill=black]
			(C) circle (2pt)node[above]{$C$};
			\draw[black, fill=black]
			(P) circle (2pt)node[right]{$S$};
			\draw[black, fill=black]
			(P1) circle (2pt);
				\draw[black, fill=black]
			(P2) circle (2pt);
			\draw[black, fill=black]
			(P3) circle (2pt);

		\end{tikzpicture}
		\caption{Simson's line.}\label{fig:simson}
	\end{center}
\end{figure}

\section{Conics}\label{sec:conics}
\subsection{Elliptical billiards}\label{sec:elliptical}

In this section, we review most important properties of billiards within an ellipse.
The most famous one is the focal property of elliptic billiards, which is illustrated in Figure \ref{fig:focal.property.ellipse}.

\begin{proposition} [First focal property of ellipses]
\label{prop:focal.property}%
Let $\mathcal E$ be an ellipse
with foci $F_1,F_2$ and $A\in\mathcal E$ an arbitrary point. Then
segments $AF_1$, $AF_2$ satisfy the billiard reflection law off $\mathcal E$.
\end{proposition}
\begin{figure}[h]
	\begin{center}
	\begin{tikzpicture}
\tikzmath{\koef=4;}
\coordinate (T2) at (2.598+\koef*0.25,1-\koef*2.598/9);
\coordinate (T1) at (2.598-\koef*0.25,1+\koef*2.598/9);

\draw[white] (T1) -- (2.598,1) coordinate (A)
-- (-2.235,0) coordinate (F1)
pic [fill=black!40, angle radius=20pt] {angle = T1--A--F1};

\draw[white] (T2) -- (A)
-- (2.235,0) coordinate (F2)
pic [fill=black!40, angle radius=20pt] {angle = F2--A--T2};

\draw[thick] (0,0) ellipse (3 and 2);
\draw[thick,dashed](F2)--(A)--(F1);

\draw[black, fill=gray, thick]
(F2) circle (2pt) node[below right] {$F_2$};
\draw[black, fill=gray, thick]
 (F1) circle (2pt) node[below right] {$F_1$};
\draw[black, fill=black]
(A) circle (2pt) node[above right] {$A$};

\draw (T1) -- (T2);

\node at (-2.5,1.5){$\mathcal{E}$};

		\end{tikzpicture}
		\caption{First focal property of
			ellipses.}\label{fig:focal.property.ellipse}
	\end{center}
\end{figure}
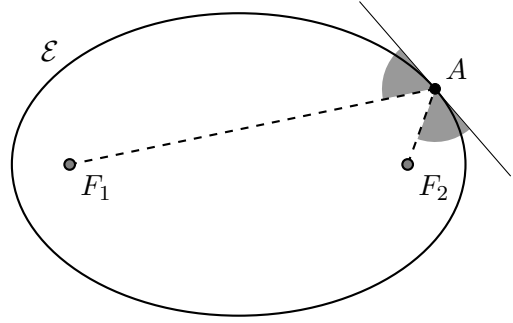

If the billiard particle is launched from one focus of the ellipse, the focal property implies that the segments of the trajectory alternately contain the foci $F_1$ and $F_2$. Next proposition characterizes the trajectories which do not contain the foci.

\begin{proposition}\label{prop:ellipse.caustic}
Suppose that the billiard particle is traveling along segment $s$ which does not contain any focus $F_1$, $F_2$ of the boundary ellipse $\mathcal{E}$.
Then we have:
\begin{itemize}
	\item If both foci of $\mathcal{E}$ are on the same side of $s$, then there is unique ellipse $\mathcal{E}'$ sharing the same foci with $\mathcal{E}$ and touching the segment. Moreover, after reflection off the boundary, the next segment of the billiard trajectory will also be tangent to $\mathcal{E}'$, see Figure \ref{fig:caustic.property.ellipse};
	\begin{figure}[h]
		\begin{center}
			\begin{tikzpicture}
				\tikzmath{\koef=4;}
				\coordinate (T2) at (2.598+\koef*0.25,1-\koef*2.598/9);
				\coordinate (T1) at (2.598-\koef*0.25,1+\koef*2.598/9);

				\draw[white] (T2) -- (2.598,1) coordinate (A)
				-- (2.68967, -0.885864) coordinate (B)
				pic [fill=black!40, angle radius=20pt] {angle = B--A--T2};
				
				\draw[white] (T1) -- (A)
				-- (-1.38561, 1.77389) coordinate (C)
				pic [fill=black!40, angle radius=20pt] {angle = T1--A--C};
				
				\draw[thick,dashed](B)--(A)--(C);

				\draw[thick] (0,0) ellipse (3 and 2);
				\draw[thick] (0,0) ellipse (2.646 and 1.414);
				\draw[black, fill=gray, thick] 	(2.235,0) circle (2pt);
				\draw[black, fill=gray, thick] 	(-2.235,0) circle (2pt);% node[below right] {$F_1$};
				\draw[black, fill=black] 	(A) circle (2pt) node[above right] {$A$};
				
				\draw (T1) -- (T2);

			\end{tikzpicture}
			\caption{Two segments satisfying the billiard reflection law of ellipse $\mathcal{E}$ are tangent to the same conic, which is confocal with $\mathcal{E}$.}\label{fig:caustic.property.ellipse}
		\end{center}
	\end{figure}
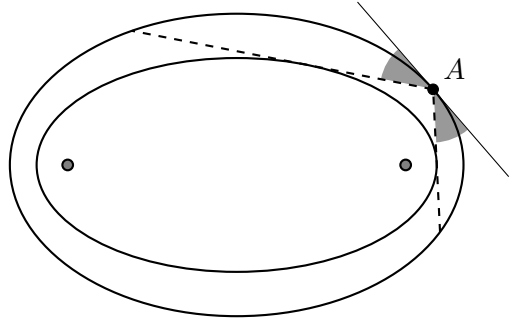

	\item If $s$ crosses the segment $F_1F_2$, then there is a hyperbola $\mathcal{H}$ sharing the same foci with $\mathcal{E}$ and touching the line containing $s$. Moreover, after reflection off the boundary, the next segment of the billiard trajectory will also be tangent to $\mathcal{H}$.
\end{itemize}
\end{proposition}

Proposition \ref{prop:ellipse.caustic} implies the following key property of elliptical billiards: each trajectory of the billiard within ellipse has \emph{a caustic} -- a curve such that each segment of the trajectory lies on a tangent line of that curve, see Figure \ref{fig:caustic}.
For elliptical billiards, caustics are ellipses  hyperbolas confocal with the boundary, including degenerate ones, which can be identified as horizontal and vertical axes of the ellipse $\mathcal{E}$.
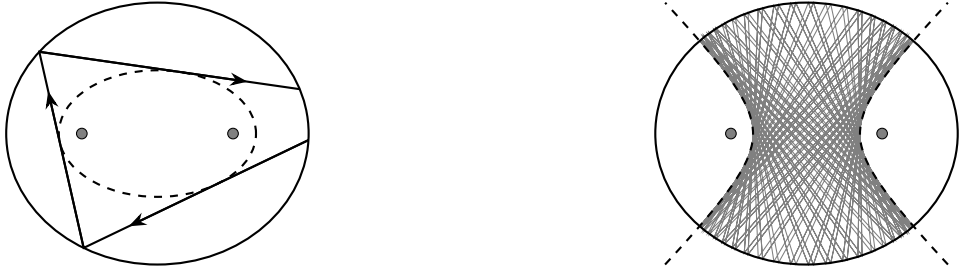
\begin{figure}[h]
	\begin{minipage}{0.5\textwidth}
		\centering
		\begin{tikzpicture}
			\tikzmath{\k1=0.8; \k2=1-\k1;}
			\coordinate (A) at (1.9975, -0.0865665);
			\coordinate (B) at (-0.975234, -1.51218);
			\coordinate (C) at (-1.5623, 1.08139);
			\coordinate (D) at (1.88146, 0.587441);	
			
			\coordinate (B1) at (-0.975234*\k1+1.9975*\k2, -1.51218*\k1-0.0865665*\k2);
			\coordinate (C1) at (-1.5623*\k1-0.975234*\k2, 1.08139*\k1-1.51218*\k2);
			\coordinate (D1) at (-1.5623*\k2+1.88146*\k1, 1.08139*\k2+0.587441*\k1);
			
			\draw[thick](A)--(B)--(C)--(D);
			\draw[thick,-Stealth](A)--(B1);
			\draw[thick,-Stealth](B)--(C1);
			\draw[thick,-Stealth](C)--(D1);

			\draw[thick] (0,0) ellipse (2 and 1.732);
			\draw[thick,dashed] ellipse (1.304 and 0.837);
			
			\draw[fill=gray] (1,0) circle (2pt);
			\draw[fill=gray] (-1,0) circle (2pt);
			
		\end{tikzpicture}
		
	\end{minipage}
	\begin{minipage}{0.5\textwidth}
		\centering
		\begin{tikzpicture}
			
			\draw[thin,gray](1.26491,1.34164)--(-0.207887,-1.72267)--(-1.07964,1.45801)--(1.40721,-1.23078)--(-0.824046,1.5782)--(-0.580162,-1.65758)--(1.36928,1.26246)--(-1.22149,-1.37149)--(0.0923642,1.7302)--(1.13967,-1.42333)--(-1.39652,1.23988)--(0.734538,-1.61101)--(0.679159,1.62913)--(-1.38794,-1.24708)--(1.17181,1.40362)--(0.0240942,-1.73193)--(-1.19312,1.39009)--(1.38085,-1.25298)--(-0.638906,1.6413)--(-0.772312,-1.5977)--(1.40153,1.23563)--(-1.11565,-1.43753)--(-0.140308,1.72778)--(1.24019,-1.35884)--(-1.36005,1.26992)--(0.537662,-1.66829)--(0.859174,1.56409)--(-1.41019,-1.22822)--(1.05286,1.47262)--(0.255106,-1.7179)--(-1.28109,1.33007)--(1.33393,-1.29053)--(-0.43148,1.69126)--(-0.939446,-1.52908)--(1.41401,1.22492)--(-0.983304,-1.50826)--(-0.367368,1.70258)--(1.31609,-1.3042)--(-1.30228,1.31455)--(0.321187,-1.70957)--(1.01297,1.49346)--(-1.41301,-1.22579)--(0.906968,1.54372)--(0.476066,-1.68227)--(-1.34542,1.28156)--(1.26486,-1.34168)--(-0.207742,1.72268)--(-1.07972,-1.45796)--(1.4072,1.23079)--(-0.823936,-1.57824)--(-0.580291,1.65754)--(1.36931,-1.26244)--(-1.22143,1.37153)--(0.0922171,-1.73021)--(1.13974,1.42329)--(-1.3965,-1.23989)--(0.734421,1.61105)--(0.679281,-1.62909)--(-1.38796,1.24707)--(1.17174,-1.40366)--(0.0242417,1.73192)--(-1.19318,-1.39005)--(1.38082,1.25299)--(-0.638781,-1.64133)--(-0.772426,1.59766)--(1.40155,-1.23562)--(-1.11558,1.43757)--(-0.140454,-1.72777)--(1.24024,1.35881)--(-1.36002,-1.26995)--(0.53753,1.66832)--(0.85928,-1.56404)--(-1.4102,1.22821)--(1.05277,-1.47267)--(0.25525,1.71789)--(-1.28114,-1.33004)--(1.33389,1.29056)--(-0.431343,-1.69129)--(-0.939544,1.52903)--(1.41401,-1.22492)--(-0.983211,1.5083)--(-0.367508,-1.70256)--(1.31613,1.30417)--(-1.30223,-1.31459)--(0.321045,1.70959)--(1.01306,-1.49341)--(-1.41301,1.22579)--(0.906867,-1.54376)--(0.476201,1.68224)--(-1.34545,-1.28153)--(1.26481,1.34171)--(-0.207597,-1.72269)--(-1.0798,1.45792)--(1.40718,-1.2308)--(-0.823827,1.57828)--(-0.58042,-1.65751)--(1.36933,1.26242)--(-1.22137,-1.37157)--(0.09207,1.73021)--(1.13981,-1.42324)--(-1.39648,1.23991)--(0.734304,-1.61109)--(0.679402,1.62905)--(-1.38798,-1.24705)--(1.17167,1.40371)--(0.0243893,-1.73192)--(-1.19325,1.39001)--(1.3808,-1.25301)--(-0.638656,1.64137)--(-0.77254,-1.59762)--(1.40156,1.2356)--(-1.1155,-1.43762)--(-0.140601,1.72777)--(1.2403,-1.35877)--(-1.35999,1.26997)--(0.537399,-1.66835)--(0.859385,1.564)--(-1.41021,-1.2282)--(1.05269,1.47271)--(0.255394,-1.71787)--(-1.28119,1.33001)--(1.33385,-1.29059)--(-0.431206,1.69131)--(-0.939641,-1.52899)--(1.41401,1.22492)--(-0.983119,-1.50835)--(-0.367648,1.70254)--(1.31617,-1.30414)--(-1.30219,1.31462)--(0.320903,-1.70961)--(1.01315,1.49337)--(-1.413,-1.22579)--(0.906767,1.5438)--(0.476336,-1.68221)--(-1.34549,1.2815)--(1.26476,-1.34175)--(-0.207452,1.72271)--(-1.07988,-1.45788)--(1.40717,1.23081)--(-0.823718,-1.57833)--(-0.580549,1.65747)--(1.36936,-1.26239)--(-1.22131,1.37161)--(0.0919229,-1.73022)--(1.13988,1.4232)--(-1.39647,-1.23992)--(0.734186,1.61113)--(0.679524,-1.62901)--(-1.388,1.24703)--(1.17161,-1.40375)--(0.0245368,1.73192)
			;
			
			\draw[thick] (0,0) ellipse (2 and 1.732);		
			\draw[fill=gray] (1,0) circle (2pt);
			\draw[fill=gray] (-1,0) circle (2pt);
			
			\draw[thick, dashed, domain=-1.732:1.731,smooth,variable=\y]
			plot ({0.7071*sqrt(1+2*\y*\y)},{\y});
			\draw[thick, dashed, domain=-1.732:1.731,smooth,variable=\y]
			plot ({-0.7071*sqrt(1+2*\y*\y)},{\y});
			
		\end{tikzpicture}
		
	\end{minipage}
	\caption{The caustics of billiard trajectories.}\label{fig:caustic}
\end{figure}

A direct consequence of Propositions \ref{prop:focal.property} and \ref{prop:ellipse.caustic} is the well-known second focal property of ellipses, illustrated in Figure \ref{fig:focal.property2}.

\begin{proposition} [Second focal property of ellipses]
	\label{prop:focal.property2}%
	Let $\mathcal E$ be an ellipse
	with foci $F_1,F_2$ and $A$ an arbitrary point outside ellipse $\mathcal E$. Denote by $t_1$ and $t_2$ the two tangents from $A$ to ellipse $\mathcal E$. Then the angle between $t_1$ and the
	segment $AF_1$ is equal to the angle between $AF_2$ and $t_2$.
\end{proposition}
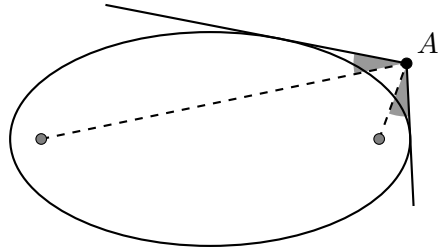
\begin{figure}[h]
	\begin{center}
		\begin{tikzpicture}
			\tikzmath{\koef=4;}
			\coordinate (T2) at (2.598+\koef*0.25,1-\koef*2.598/9);
			\coordinate (T1) at (2.598-\koef*0.25,1+\koef*2.598/9);
			
			\draw[white] (2.235,0) coordinate (F2) -- (2.598,1) coordinate (A) --  (2.68967, -0.885864) coordinate (B)
			pic [fill=black!40, angle radius=20pt] {angle = F2--A--B};
			
			\draw[white] (-2.235,0) coordinate (F1) -- (A) -- (-1.38561, 1.77389) coordinate (C)
			pic [fill=black!40, angle radius=20pt] {angle = C--A--F1};
			
			\draw[thick,dashed](F1)--(A)--(F2);
			\draw[thick](B)--(A)--(C);

			%		\draw[thick] (0,0) ellipse (3 and 2);
			\draw[thick] (0,0) ellipse (2.646 and 1.414);
			\draw[black, fill=gray] 	(F2) circle (2pt);% node[below right] {$F_2$};
			\draw[black, fill=gray] 	(F1) circle (2pt);% node[below right] {$F_1$};
			\draw[black, fill=black] 	(A) circle (2pt) node[above right] {$A$};

		\end{tikzpicture}
		\caption{Second focal property of
			ellipses}\label{fig:focal.property2}
	\end{center}
\end{figure}

In the next lemma, we note that the type of the caustic may be
determined by the period of a trajectory.
\begin{lemma}\label{lemma:hyperbola} A periodic trajectory of the elliptic billiard with odd period has an ellipse as the caustic.
\end{lemma}
\begin{proof} Suppose that the caustic of a given trajectory is a hyperbola. Denote by $F_1, F_2$ the focal points of the boundary ellipse. Then every segment of the billiard trajectory intersects the segment $F_1F_2$. Thus, for a periodic trajectory there will be an even number of intersections of the trajectory with $F_1F_2$. Consequently, the period of a closed trajectory with hyperbola as caustic must be even.
\end{proof}

\subsection{Pascal's theorem}
We review here the classical Pascal's theorem.
For more details, see, for example \cite{BergerGeomII} or \cite{DragRadn2011book}.

\begin{theorem}[Pascal's theorem ]\label{th:pascal}
Let a non-degenerate conic be given, and $M$, $N$, $O$, $P$, $Q$, $R$ six points on that conic. 
Then the intersection points of lines $MN$ and $PQ$, $NO$ and $QR$, $OP$ and $RM$ are collinear.
(See Figure \ref{fig:pascal}).
\end{theorem}
\begin{figure}[h]
	\begin{center}
		\begin{tikzpicture}
			\coordinate(M) at (2.23607,-2);
			\coordinate(N) at (-1,0);
			\coordinate(O) at (1.41421,1);
			\coordinate(P) at (-1.80278,-1.5);
			\coordinate(R) at (-2.69258,2.5);
			\coordinate(Q) at (1.22066,-0.7);
			
			\coordinate(p1) at (0.0831955, -0.0343633);
			\coordinate(p2) at (0.4588, -0.901588);
			\coordinate(p3) at (-0.0941886, 0.3752);
			
			\draw[very thick,black](p2)--(p3);
			
			\draw[thin,gray](M)--(N)--(O)--(P)--(Q)--(R)--cycle;

					\draw [black, very thick, %dashed,
			domain=-3:3, samples=40] 
			plot ({sqrt(1+\x*\x)}, {\x} );
			
			\draw [black, very thick,   domain=-3:3, samples=40] 
			plot ({-1*sqrt(1+\x*\x)}, {\x} );

\draw[fill=black](M) circle (2pt) node[right]{$M$};
\draw[fill=black](O) circle (2pt) node[right]{$O$};
\draw[fill=black](Q) circle (2pt) node[right]{$Q$};

\draw[fill=black](N) circle (2pt) node[left]{$N$};
\draw[fill=black](P) circle (2pt) node[left]{$P$};
\draw[fill=black](R) circle (2pt) node[below left]{$R$};

\draw[fill=black,black](p1) circle (2pt);
\draw[fill=black,black](p2) circle (2pt);
\draw[fill=black,black](p3) circle (2pt);

		\end{tikzpicture}
		\caption{Pascal's theorem.}\label{fig:pascal}
	\end{center}
\end{figure}
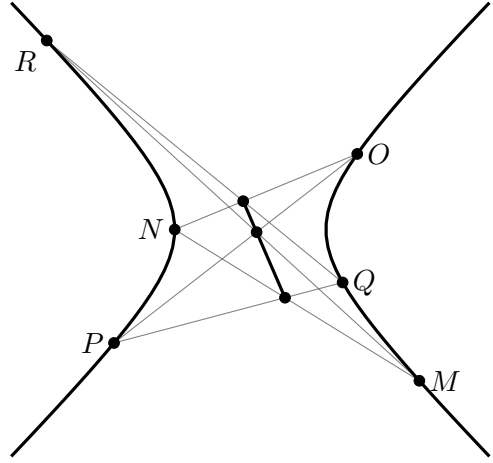

Pascal's theorem can be also applied in the cases when some of the six given points on the conic coincide.
In such cases, the line containing two coinciding points is the tangent line to the conic at that point.
Two instances of such limit cases of Pascal's theorem are provided below.

\begin{corollary}\label{cor:pascal1}
Let a non-degenerate conic be given together with its four
points $M$, $N$, $P$, $Q$.
The intersection points of the tangents to the conic
at $M$ and $P$ is collinear with the points of intersection of the pairs of lines $MN$ and $PQ$, $NP$ and $MQ$.
(See Figure \ref{fig:pascal2}).
\end{corollary}
\begin{figure}[h]
	\begin{center}
		\begin{tikzpicture}
			\coordinate(M) at (2.23607,-2);
			\coordinate(N) at (-1,0);
			\coordinate(P) at (-1.80278,-1.5);
			\coordinate(Q) at (1.22066,-0.7);
			
			\coordinate(p1) at (-0.319404, 1.2717);
			\coordinate(p2) at (-0.0718451, 0.58033);
			\coordinate(p3) at (0.4588, -0.901588);
			
			\draw[very thick,black](p1)--(p3);
			
			\draw[thin,gray](M)--(N);
			\draw[thin,gray](P)--(Q);
			\draw[thin,gray](P)--(p1)--(M)--(p2)--cycle;

			\draw [black, very thick, %dashed,
			domain=-3:3, samples=40] 
			plot ({sqrt(1+\x*\x)}, {\x} );
			
			\draw [black, very thick,   domain=-3:3, samples=40] 
			plot ({-1*sqrt(1+\x*\x)}, {\x} );

			\draw[fill=black](M) circle (2pt) node[right]{$M$};
			\draw[fill=black](Q) circle (2pt) node[right]{$Q$};

			\draw[fill=black](N) circle (2pt) node[left]{$N$};
			\draw[fill=black](P) circle (2pt) node[left]{$P$};
			
			\draw[fill=black,black](p1) circle (2pt);
		\draw[fill=black,black](p2) circle (2pt);
			\draw[fill=black,black](p3) circle (2pt);

		\end{tikzpicture}
		\caption{Pascal's theorem: a variant.}\label{fig:pascal2}
	\end{center}
\end{figure}
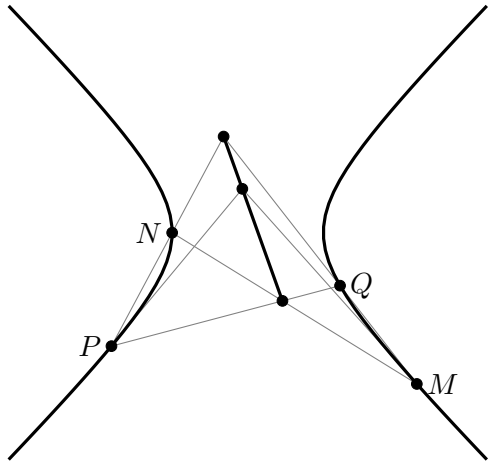

\begin{corollary}\label{cor:simpson}
Let $M$, $N$, $P$ be three points on a non-degenerate conic. Then the intersections of the tangent lines to the conic at $M$, $N$, $P$ with the opposite sides of the triangle $MNP$ are collinear. (See Figure \ref{fig:simpson}).
\end{corollary}
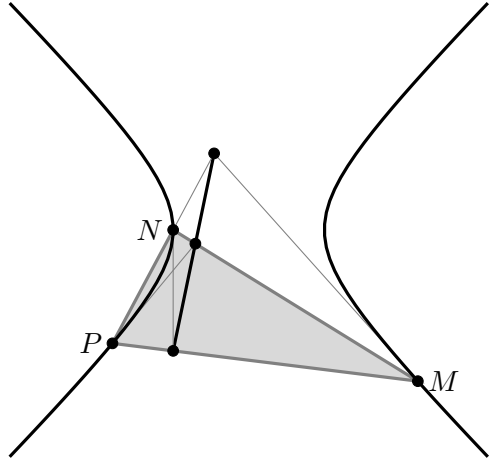
\begin{figure}[h]
	\begin{center}
		\begin{tikzpicture}
			\coordinate(M) at (2.23607,-2);
			\coordinate(N) at (-1,0);
			\coordinate(P) at (-1.80278,-1.5);
			
			\coordinate(p1) at (-0.705929, -0.181746);
			\coordinate(p2) at (-0.458223, 1.01231);
			\coordinate(p3) at (-1, -1.59938);
			
			\draw[very thick,gray,fill=gray!30](M)--(N)--(P)--cycle;
			
			\draw[very thick,black](p2)--(p3);
			
			\draw[gray](N)--(p3);
			\draw[gray](P)--(p1);
			\draw[gray](M)--(p2)--(N);

			\draw [black, very thick, %dashed,
			domain=-3:3, samples=40] 
			plot ({sqrt(1+\x*\x)}, {\x} );
			
			\draw [black, very thick,   domain=-3:3, samples=40] 
			plot ({-1*sqrt(1+\x*\x)}, {\x} );

			\draw[fill=black](M) circle (2pt) node[right]{$M$};

			\draw[fill=black](N) circle (2pt) node[left]{$N$};
			\draw[fill=black](P) circle (2pt) node[left]{$P$};
			
			\draw[fill=black,black](p1) circle (2pt);
			\draw[fill=black,black](p2) circle (2pt);
			\draw[fill=black,black](p3) circle (2pt);

		\end{tikzpicture}
		\caption{Corollary \ref{cor:simpson}.}\label{fig:simpson}
	\end{center}
\end{figure}

The classical Pappus theorem is a version of Pascal's theorem for degenerate conics.

\section{The Siebeck--Marden theorem}\label{sec:marden}

Marden's theorem is one of the fundamental results in geometric theory of polynomials and
rational functions.
That theorem has a long history which is well described in the famous Marden book
\cite{MardenPOLY}.
The earliest version of this theorem, up to our best knowledge, goes back to 1864 when Siebeck (see \cite{Si1864})
formulated and proved it for the case of polynomials with simple
roots.
Here we consider only the case $n=3$.

\begin{theorem}[Siebeck, \cite{Si1864}, $n=3$]\label{th:Siebeck}
Let $P(z)$ be a polynomial of degree $3$
with complex coefficients, such that its zeros
$\alpha_1$, $\alpha_2$, $\alpha_3$ are simple and noncollinear.
Then there exists an ellipse $\mathcal E$ tangent to every
line segment $[\alpha_i, \alpha_j]$ at the midpoint, see Figure \ref{fig:siebeck}.
Moreover, the foci of the
curve $\mathcal E$ are zeros of the derivative polynomial
$P'(z)$.
\end{theorem}
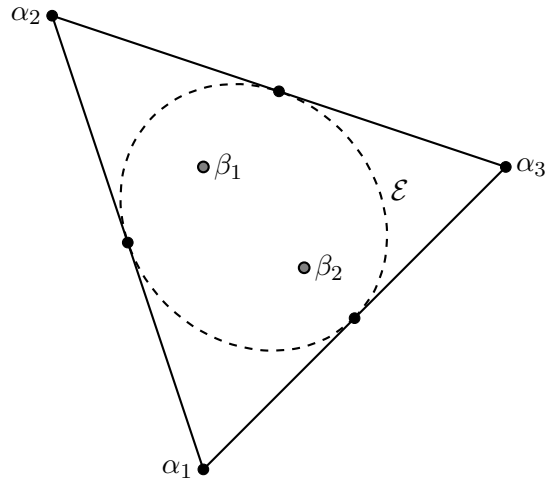
\begin{figure}[h]
\begin{center}
\begin{tikzpicture}
\draw[thick](2,0)--(0,6)--(6,4)--cycle;
\draw[thick,fill=gray] (2,4) circle (2pt) node[right]{$\beta_1$};
\draw[thick,fill=gray] (10/3,8/3) circle (2pt) node[right]{$\beta_2$};
\draw[fill=black](2,0) circle (2pt) node[left]{$\alpha_1$};
\draw[fill=black](0,6) circle (2pt) node[left]{$\alpha_2$};
\draw[fill=black](6,4) circle (2pt) node[right]{$\alpha_3$};
\draw[fill=black](1,3) circle (2pt);
\draw[fill=black](4,2) circle (2pt);
\draw[fill=black](3,5) circle (2pt);

\draw[thick, dashed, domain=0:360,smooth,variable=\t]
plot ({8/3+1.88562*0.7071*cos(\t)+1.63299*0.7071*sin(\t)},
{10/3-1.88562*0.7071*cos(\t)+1.63299*0.7071*sin(\t)});

\node at (4.6,3.7){$\mathcal{E}$};

\end{tikzpicture}
\caption{Siebeck theorem. If $\alpha_1$, $\alpha_2$, $\alpha_3$ are zeros of a cubic polynomial $P(z)$ and $\beta_1$, $\beta_2$ zeros of its derivative $P'(z)$, then there is an ellipse with foci $\beta_1$, $\beta_2$ inscribed in triangle with vertices $\alpha_1$, $\alpha_2$, $\alpha_3$. Moreover, the ellipse is touching the sides at midpoints.}\label{fig:siebeck}
\end{center}
\end{figure}

We note that the ellipse which touches the midpoints of the triangle sides  is 
\emph{the Steiner ellipse}, i.e.~the ellipse
of the maximal area inscribed in that triangle. It is interesting to note that the ratio of areas of any triangle and its Steiner ellipse equals $3\sqrt{3}/{(4\pi)}$, see e.g.~\cite{GSO}.

 The nontrivial and interesting result of Siebeck still attracts a lot of attention, see for example \cites{Kal2008, Drag2011}. In \cite{Pra}, this theorem is attributed to van den Berg, and two proofs are presented: the original proof of van der Berg from 1888 \cite{Berg} and another one from \cite{Scho}.
We note that the main points of the proof of Siebeck theorem for $n=3$
from \cite{Kal2008} are based on  Propositions
\ref{prop:focal.property} and
\ref{prop:ellipse.caustic}, which indicates a deep relationship with elliptic billiards.
We are going to exploit that connection in next Section \ref{sec:3periodic}.

Siebeck's theorem can be extended to the cases with non-simple roots, namely to the functions of the form:
\begin{equation}\label{eq:polynomial}
P(z)=(z-\alpha_1)^{m_1} (z-\alpha_2)^{m_2}(z-\alpha_3)^{m_3}.
\end{equation}
Each zero of the derivative $P'(z)$ is then either equal to some $\alpha_i$ such that $m_i>1$ or it is a zero of the logarithmic derivative of $P$:
\begin{equation}\label{eq:function0}
F(z)=(\log P(z))'=\frac{m_1}{z-\alpha_1}+\frac{m_2}{z-\alpha_2} +\frac{m_3}{z-\alpha_3}.
\end{equation}
Notice that $F(z)$ has two zeros.
The following statement then holds, see also Figure \ref{fig:marden}.

\begin{theorem}[\cite{MardenPOLY}*{Theorem 4.2 for $n=3$}]\label{th:marden}
Consider a function $P(z)$ of the form \eqref{eq:polynomial}, such that $\alpha_1$, $\alpha_2$, $\alpha_3$ are noncollinear complex numbers, and $m_1$, $m_2$, $m_3$ non-zero real constants.
Let $\beta_1$, $\beta_2$ are the zeros of the logarithmic derivative \eqref{eq:function0} of $P(z)$.
Then there is a conic with foci $\beta_1$, $\beta_2$ touching each line of the segment $[\alpha_i, \alpha_j]$
in a point dividing that segment in the ratio $m_i:m_j$.
\end{theorem}
\begin{figure}[h]
	\begin{center}
		\begin{tikzpicture}
			\draw[thick](2,0)--(0,6)--(6,4)--cycle;
			\draw[thick,fill=gray] (2.97165, 4.35634) circle (2pt) node[below right]{$\beta_1$};
			\draw[thick,fill=gray] (2.40955, 1.64366) circle (2pt) node[above right]{$\beta_2$};
			\draw[fill=black](2,0) circle (2pt) node[left]{$\alpha_1$};
			\draw[fill=black](0,6) circle (2pt) node[left]{$\alpha_2$};
			\draw[fill=black](6,4) circle (2pt) node[right]{$\alpha_3$};
			\draw[fill=black](1.33333, 2) circle (2pt);
			\draw[fill=black](3.4641, 1.4641) circle (2pt);
			\draw[fill=black](3.21539, 4.9282) circle (2pt);
			
			\draw[thick, dashed, domain=0:360,smooth,variable=\t]
			plot ({2.6906+2.0018*0.2029*cos(\t)-1.44518*0.979199*sin(\t)},
			{3+1.44518*0.2029*sin(\t)+2.0018*0.979199*cos(\t)});
			
			%\node at (4.6,3.7){$\mathcal{C}$};
			
		\end{tikzpicture}
		\caption{Marden theorem. An example for $m_1=1$, $m_2=2$, $m_3=\sqrt3$.}\label{fig:marden}
	\end{center}
\end{figure}
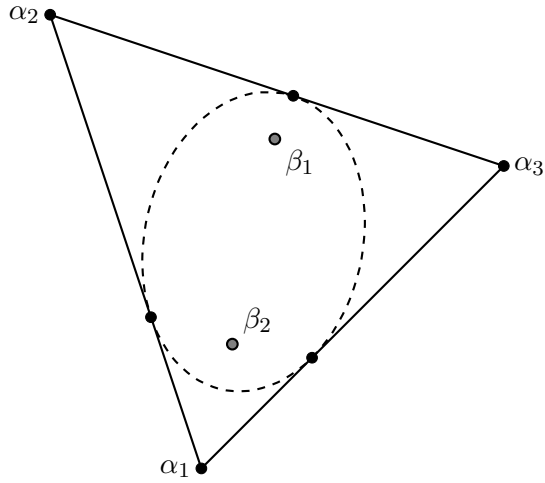

Now we are going to expand on those results from the geometric theory of polynomials, in order to give complete characterization of ellipses inscribed in triangles.

\begin{lemma}\label{lem:focusfocus}
Suppose that $F_1$ is a point inside a triangle.
Then there exists a unique ellipse inscribed in the triangle, such that one of its foci is point $F_1$.
\end{lemma}

\begin{proof} Denote the triangle by $ABC$. Denote by $F_1'$, $F_1''$, $F_1'''$ the points symmetric to $F_1$  with respect to
sides $BC$, $AC$, $AB$ respectively. Then the second focal point of the ellipse inscribed in $\triangle ABC$ is obtained as the intersection of bisectors of angles $F_1''AF_1'''$ and $F_1''CF_1'$.  Namely, if $F_2$ is a second focal point, then triangle $CF_1''F_2$ is congruent to triangle $CF_1'F_2$. Thus, angle $F_2CF_1''$ is equal to angle $F_2CF_1'$.
\end{proof}

\begin{theorem}\label{thm:ellipseconc}
Let $ABC$ be a given triangle, and $\mathcal{E}$ an ellipse inscribed in it, touching the sides $BC$, $AC$, $AB$ in points $K$, $L$, $M$ respectively.
Then the lines $AK$, $BL$, $CM$ are concurrent.
(See Figure \ref{fig:ellipseconc}.)
\end{theorem}
\begin{figure}[h]
	\begin{center}
		\begin{tikzpicture}
						\draw[gray](2,0)--(3.21539, 4.9282);
			\draw[gray](0,6)--(3.4641, 1.4641);
			\draw[gray](6,4)--(1.33333, 2);

			\draw[thick](2,0)--(0,6)--(6,4)--cycle;
			\draw[fill=black](2,0) circle (2pt) node[left]{$A$};
			\draw[fill=black](0,6) circle (2pt) node[left]{$B$};
			\draw[fill=black](6,4) circle (2pt) node[right]{$C$};
			\draw[fill=black](1.33333, 2) circle (2pt) node[left]{$M$};
			\draw[fill=black](3.4641, 1.4641) circle (2pt)node[below right]{$L$};
			\draw[fill=black](3.21539, 4.9282) circle (2pt)node[above]{$K$};
			
			\draw[thick, dashed, domain=0:360,smooth,variable=\t]
			plot ({2.6906+2.0018*0.2029*cos(\t)-1.44518*0.979199*sin(\t)},
			{3+1.44518*0.2029*sin(\t)+2.0018*0.979199*cos(\t)});

		\end{tikzpicture}
		\caption{The lines connecting triangle vertices with the points where the inscribed ellipse touches the opposite sides are concurrent.}\label{fig:ellipseconc}
	\end{center}
\end{figure}
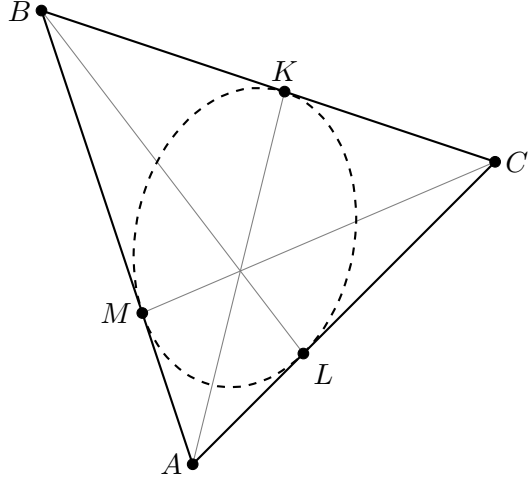
\begin{proof}
Let $\varphi$ be an affine transformation that maps the ellipse $\mathcal{E}$ to a circle $\mathcal{E}'$.
That circle is inscribed in triangle $\varphi(ABC)$.

Since segments of each pair $\varphi(A)\varphi(L)$ and $\varphi(A)\varphi(M)$, $\varphi(B)\varphi(K)$ and $\varphi(B)\varphi(M)$, $\varphi(C)\varphi(K)$ and $\varphi(C)\varphi(L)$ are of equal lengths as tangent segments to the circle from the same vertex, the Ceva's theorem (Theorem \ref{thm:ceva}) gives that lines $\varphi(AK)$, $\varphi(BL)$, $\varphi(CL)$ are concurrent.

The statement follows immediately.
\end{proof}

In order to align further this matter with the Marden theorem, we provide an alternative proof of the last theorem.

\begin{proof}[Second proof of Theorem \ref{thm:ellipseconc}] 
Let $F_1$ be one of the foci of $\mathcal{E}$. Denote the complex numbers corresponding to $A$, $B$, $C$, and $F_1$ by $z_1$, $z_2$, $z_3$, and $q$ respectively.
We search for nonzero real numbers $m_1$, $m_2$, $m_3$ such that $q$ is one of the zeros of the derivative of $(z-z_1)^{m_1} (z-z_2)^{m_2} (z-z_3)^{m_3}$. 
This imposes one complex relation on real numbers $m_1$, $m_2$, $m_3$. 
That complex relation gives two real relations which uniquely determine three real numbers $m_1$, $m_2$, $m_3$ up to a nonzero  real factor. 
According to Marden's theorem (Theorem \ref{th:marden}), there is an ellipse $\mathcal{E}_1$ inscribed in triangle $ABC$ which has $F_1$ as one of its foci, with the other focus corresponding to the second zero of the derivative.
Since ellipses $\mathcal{E}$ and $\mathcal{E}_1$ are both inscribed in $\triangle ABC$ and both have $F_1$ as one of their foci, Lemma \ref{lem:focusfocus} implies that they coincide.
\end{proof}

 \begin{remark} One can compare the last theorem with the Bradley theorem, see e.g. \cite{Bar}.
 \end{remark}
	
\begin{theorem}\label{thm:ellipsemarden}
Every ellipse inscribed in a triangle is a Marden ellipse, i.e.~there exist positive real numbers
$m_1$, $m_2$, $m_3$ such that the foci of the ellipse are the zeros of the logarithmic derivative of the function
$(z-\alpha_1)^{m_1} (z-\alpha_2)^{m_2} (z-\alpha_3)^{m_3}$, with $\alpha_1$, $\alpha_2$, $\alpha_3$ being the complex numbers corresponding to the vertices of the triangle.
The numbers $m_1$, $m_2$, $m_3$ are unique up to a non-zero factor.
	\end{theorem}
	
\begin{proof} Denote the triangle by $ABC$, the inscribed ellipse by $\mathcal{E}$, and by $K$, $L$, $M$ the common points of $\mathcal{E}$ with sides $BC$, $AC$, $AB$ respectively.
According to Theorem \ref{thm:ellipseconc}, the lines $AK$, $BL$, $CM$ are concurrent, thus Ceva's theorem implies that \eqref{eq:ceva} is satisfied.
This relation determines positive numbers $m_1$, $m_2$, $m_3$ uniquely up to a non-zero factor.
	\end{proof}

\section{$3$-periodic trajectories of billiards within ellipses}\label{sec:3periodic}

Now, after reviewing billiards and Marden's theory in previous sections, we are equipped and ready to address the the title question.
%Givan an arbitrary triangle, does there exist an ellipse for which this triangle is a billiard trajectory?

\begin{theorem}\label{thm:triangles} Every triangle  is a $3$-periodic trajectory of the billiard within an ellipse.
That ellipse is uniquely determined.
\end{theorem}
\begin{proof}
Let an arbitrary triangle $KLM$ be given and
construct a new triangle $ABC$, formed by the bisectors of the exterior angles of $KLM$, which implies that $KLM$ is a $3$-periodic billiard trajectory within triangle $ABC$.

According to Theorem \ref{th:acute}, points $K$, $L$, $M$ are the feet of the altitudes of triangle $ABC$, so the three lines $AK$, $BL$, $CM$ intersect at one point -- the orthocenter of triangle $ABC$.
Thus, according to  Ceva's Theorem (see Theorem \ref{thm:ceva}),
there exist nonzero real numbers $m_1$, $m_2$, $m_3$ such that point $K$ divides $BC$ in the ratio $m_2:m_3$, point $L$ divides $CA$ in the ratio $m_3:m_1$, and $M$ divides $AB$ in the ratio $m_1:m_2$. The numbers  $m_1$, $m_2$, $m_3$ can be calculated in terms of the lengths
of the sides of triangle $ABC$, see Lemma \ref{lem:ms} and relation \eqref{eq:ms}.

Now, according to  Marden's theorem (Theorem  \ref{th:marden}), there exists a conic $\mathcal{E}$ which touches the sides $AB$, $BC$, $CA$ respectively at $M$, $K$, $L$.
Since $K$, $L$, $M$ are inner points of those sides, $\mathcal{E}$ must be an ellipse inscribed in the triangle $ABC$, see Figure \ref{fig:marden}.
By construction, $KLM$ is a billiard trajectory
within ellipse $\mathcal{E}$.
\end{proof}

We provide also another proof, based on theorems of Pascal and Menelaus.

\begin{proof}[Second proof of Theorem \ref{thm:triangles}]
Let $KLM$ be a triangle, lines $k$, $l$, $m$ the bisectors of its exterior angles at the vertices $K$, $L$, $M$ respectively, and $P$, $Q$, $R$ the intersections of those lines with $LM$, $KM$, $KL$ respectively.
According to Corrolary \ref{cor:exterior-angles}, points $P$, $Q$, $R$ are collinear.

Now, there is a unique conic touching the lines $k$, $l$ at $K$, $L$ and containing $M$.
Let $m'$ be the tangent line to that conic at $M$ and denote $R'=m'\cap KL$.
According to Corollary \ref{cor:simpson}, $P$, $Q$, $R'$ are collinear, which gives $R=R'$, i.e.~$m=m'$.
Since the conic is touching the sides of the triangle determined by $k$, $l$, $m$ in inner points $K$, $L$, $M$, it must be an ellipse. 
\end{proof}

%Theorem \ref{thm:triangles} can be reformulated as a statement that every triangle is a Poncelet polygon inscribed in and circumscribed about a pair of ellipses which share both foci.

%For quadrilaterals we can state the following analog:

%\begin{theorem} Every convex quadrilateral without three colinear vertices is a Poncelet polygon inscribed in and circumscribed about a pair of ellipseswhich share one focal point.
%\end{theorem}

%\begin{proof} Follows from the fact that every nondegenerate convex quadrilateral is a projective image of a square. Then we apply Theorem from \cite{Ts}.
%\end{proof}

\section{Convex $4$-periodic trajectories of billiards within ellipses}\label{sec:convex4}

\subsection{Is every parallelogram a trajectory of an elliptical billiard?}\label{sec:parallel}

In this section, in Theorem \ref{th:parall}, we will show that being a convex $4$-periodic elliptical billiard trajectory represents a complete characterization of parallelograms.

We start by formulating the following useful statement.

\begin{lemma}\label{lemma:ellipse-rectangle}
Let a $4$-periodic convex billiard trajectory within an ellipse be given, see Figure \ref{fig:4-periodic}.
Then we have:
\begin{itemize}
	\item its caustic is an ellipse;
	\item that trajectory is a parallelogram;
	\item the tangent lines to the boundary ellipse at the reflection points form a rectangle;
	\item the parallelogram is also a closed billiard trajectory within the rectangle;
	\item the diagonals of the rectangle are parallel to the sides of the parallelogram.
\end{itemize} 
\begin{figure}[h]
	\begin{center}
		\begin{tikzpicture}[]
			\coordinate (E) at (-1.75716, 0.985292);
			\coordinate (F) at (-2.31683, -0.45909);
			\coordinate (G) at (1.75716, -0.985292);
			\coordinate (H) at (2.31683, 0.45909);
			
			\coordinate (A) at (-2.80513,0.362314);
			\coordinate (B) at (-1.02193, -2.63736);
			\coordinate (C) at (2.80513,-0.362314);
			\coordinate (D) at (1.02193, 2.63736);

			\draw[very thick,gray](0,0) ellipse (2.44949 and 1.41421);

			\draw[thick](E)--(F)--(G)--(H)--cycle;

			\draw[fill=black](E) circle (2pt) %node[above]{$E$}
			;
			\draw[fill=black](F) circle (2pt) 
			%node[below left]{$F$}
			;
			\draw[fill=black](G) circle (2pt)
			% node[below right]{$G$}
			;
			\draw[fill=black](H) circle (2pt) %node[right]{$H$}
			;
			
			\draw(A)--(B)--(C)--(D)--cycle;
			
			\draw[dashed](A)--(C);
			\draw[dashed](B)--(D);
			
			%			\draw[fill=black](A) circle (2pt) node[left]{$A$};
			%			\draw[fill=black](B) circle (2pt) node[below]{$B$};
			%			\draw[fill=black](C) circle (2pt) node[right]{$C$};
			%			\draw[fill=black](D) circle (2pt) node[above]{$D$};
			
		\end{tikzpicture}
		\caption{A convex $4$-periodic trajectory within ellipse and the tangent lines at the points of reflection.}\label{fig:4-periodic}
	\end{center}
\end{figure}
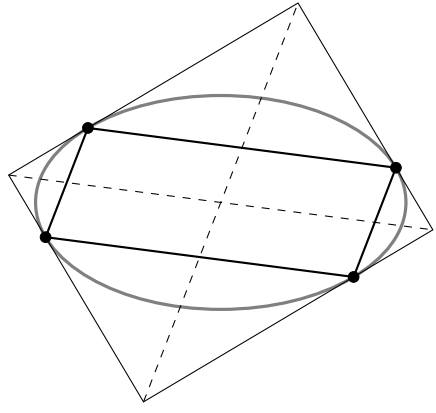
\end{lemma}
\begin{proof}
The first statement follows from the fact that any trajectory with hyperbola as caustic has self-intersections.
Then, the second statement is obtained since any periodic trajectory with ellipse as caustic is symmetric with respect to the center, see for example \cite{DR2019}*{Theorem 4}.
The rest is obtained by straightforward calculation of the angles and application of the billiard reflection law.
\end{proof}

\begin{lemma}\label{lemma:parall-rectangle}
Any parallelogram is a closed billiard trajectory within a unique rectangle. 

Moreover, for any rectangle and any point on one of its sides, there is a unique $4$-periodic billiard trajectory within the rectangle such that the given point is one of its vertices.
\end{lemma}
\begin{proof}
The sides of the rectangle are bisectors of the exterior angles of the parallelogram.

The second part follows from the fact that the segments of any $4$-periodic billiard trajectory within a rectangle are parallel to its diagonals, see Figure \ref{fig:4-periodic}.
\end{proof}

\begin{theorem}\label{th:parall}
Any parallelogram is a closed billiard trajectory within an ellipse. Moreover, such an ellipse is uniquely determined for a given parallelogram.
\end{theorem}

We immediately provide a compact, synthetic proof of this theorem, based on the Pascal theorem (see Theorem \ref{th:pascal}). An alternative, analytic proof is given in Section \ref{sec:analytic-parallelogram}.

\begin{proof}
Let $EFGH$ be a given parallelogram.
According to Lemma \ref{lemma:parall-rectangle}, there is a unique rectangle within which $EFGH$ is a closed billiard trajectory, see Figure \ref{fig:4-periodic-diag}.

\begin{figure}[h]
	\begin{center}
		\begin{tikzpicture}[]
			\coordinate (E) at (1, 0);
			\coordinate (F) at (4, 2.25);
			\coordinate (G) at (3,3);
			\coordinate (H) at (0,0.75);
			
			\coordinate (A) at (0,0);
			\coordinate (B) at (4,0);
			\coordinate (C) at (4,3);
			\coordinate (D) at (0,3);
			
			\coordinate(X) at (7,0);
			\coordinate(Y) at (4,-2.25);
			
			\draw[gray](F)--(X)--(B);
			\draw[gray](E)--(Y)--(X);
			
			\draw[very thick](A)--(B)--(C)--(D)--cycle;
			\draw[thick,gray,dashed](A)--(C);
			\draw[thick,gray,dashed](B)--(D);

			\draw[thick](E)--(F)--(G)--(H)--cycle;
			
			\draw[fill=black](E) circle (2pt) node[below]{$E$}
			;
			\draw[fill=black](F) circle (2pt)
			node[right]{$F$}
			;
			\draw[fill=black](G) circle (2pt)
			node[above]{$G$}
			;
			\draw[fill=black](H) circle (2pt) node[left]{$H$}
			;
			\draw[fill=black](X) circle (2pt) node[below]{$X$}
			;
			\draw[fill=black](Y) circle (2pt) node[below]{$Y$}
			;
			
			\node at (3,-.3){$a$};
			\node at (4.2,1){$b$};

		\end{tikzpicture}
		\caption{A $4$-periodic billiard trajectory within a rectangle.}\label{fig:4-periodic-diag}
	\end{center}
\end{figure}
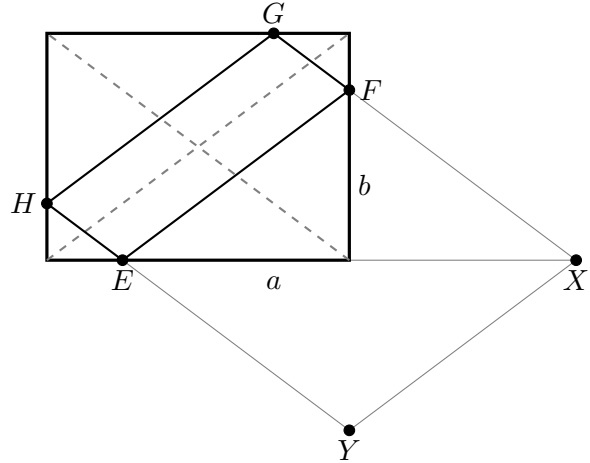
Denote by $a$ the side of the rectangle containing point $E$ and by $\mathcal{C}$ the unique conic circumscribed about $EFGH$ such that $a$ is its tangent line.
We want to prove that the conic is touching the remaining sides of the rectangle exactly at the points $F$, $G$, $H$.

Applying Pascal's theorem to $\mathcal{C}$ and its points $E$, $E$, $F$, $F$, $G$, $H$, (see Corollary \ref{cor:pascal1}) we get that the intersection points of the following pairs of lines are collinear:
$a$ and $FG$; $EF$ and $GH$; $b'$ and $EH$,
where $b'$ is the tangent line to $\mathcal{C}$ at $F$.
Denote $X=a\cap FG$, $Y=b'\cap EH$.
Since $EF$ and $GH$ are parallel, this implies that the line $XY$ also must be parallel to them, thus $Y$ is the intersection point of $EH$ with the line which is parallel to $EF$ and contains $X$.

Parallelogram $EFXY$ is rhombus, since $EF=FX$, so its diagonals are orthogonal to each other, which implies that $Y\in b$, where $b$ is the side of the rectangle containing point $F$.
Thus $b=b'$.

Similarly, we can prove that the remaining two sides of the rectangle are also tangent to $\mathcal{C}$.

It remains to show that $\mathcal{C}$ is an ellipse.
Each degenerate conic which contains points $E$, $F$, $G$, $H$ is intersecting transversely the sides of the rectangle, thus $\mathcal{C}$ must be non-degenerate.
No two tangent lines of a parabola are parallel, thus $\mathcal{C}$ is either hyperbola or ellipse.
For a hyperbola, we can notice that a pair of parallel lines cannot touch the same branch, and that the touching points of the lines containing sides of the rectangle would be placed on the extensions of those sides, which does not correspond to the geometric arrangement that we got here.

We conclude that $\mathcal{C}$ is the unique ellipse within which parallelogram $EFGH$ is a closed billiard trajectory.
\end{proof}

\subsection{An analytic proof}\label{sec:analytic-parallelogram}
	
In the previous Section \ref{sec:parallel}, we gave a synthetic proof of Theorem \ref{th:parall}.
In this section, we provide another proof.
To prepare for that, we need the following three lemmas.
	
	\begin{lemma}\label{lem:OX} Consider an ellipse with foci $F_1$, $F_2$ and center $O$.
Let $M$ be a point on that ellipse, $t$ the line touching ellipse at $M$, $F_1'$ the point symmetric to $F_1$ with respect to $t$, and $X$ the projection of $F_1$ to $t$, as shown in Figure \ref{fig:OX}.
Then we have:
\begin{itemize}
	\item Lines $F_1'M$ and $OX$ are parallel and $OX=\dfrac{1}{2}F_1'F_2$.
	\item Therefore, the length of the segment $OX$ does not depend on the choice of point $M$ on the ellipse, since $d=F_1M+F_2M=F_1'F_2=2OX$, where $d$ is the defining ``rope length" for the ellipse.
	\item Moreover, the line contaning point $M$ which is orthogonal to $t$ and the line containing point $F_1$ which is parallel to $t$ intersect each other at a point on $OX$.
\end{itemize}
	\end{lemma}
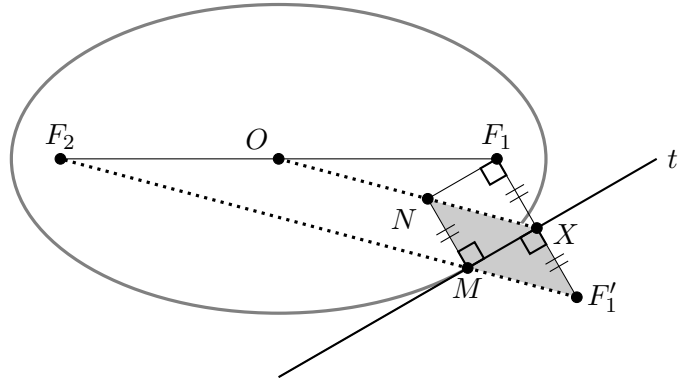
\begin{figure}[h]
	\centering
	\begin{tikzpicture}[]
		\coordinate (A) at (0,-2.88675);
		\coordinate (B) at (5,0);
			\coordinate (C) at (0,-2.88675);
		\coordinate (M) at (2.5,-1.44338);
		\coordinate (F1) at (2.88675,0);
		\coordinate (F2) at (-2.88675,0);
		\coordinate (O) at (0, 0);
		\coordinate(X) at (3.41506, -0.915064);
		\coordinate(N) at (1.97169, -0.528312);
		\coordinate(F1') at (3.94338, -1.83013);

		\draw[very thick,gray](0,0) ellipse (3.53553 and 2.04124);
			\draw[white,fill=gray!40,opacity=0.8](F1')--(X)--(N)--(M)--cycle;

		\draw[thick](A)--(B)node[right]{$t$};
		\draw(C)--(M);
		\draw(F2)--(F1)--(F1');
		\draw(F1)--(N)--(M);
		\draw[very thick,dotted](O)--(X);
		\draw[very thick,dotted](F1')--(F2);

		\draw[fill=black](M) circle (2pt) node[below]{$M$};
		\draw[fill=black](F1) circle (2pt) node[above]{$F_1$};
		\draw[fill=black](F2) circle (2pt) node[above]{$F_2$};
		\draw[fill=black](F1') circle (2pt) node[right]{$F_1'$};
		\draw[fill=black](O) circle (2pt) node[above left]{$O$};
		\draw[fill=black](X) circle (2pt);
		\node at (3.8, -1){$X$};			
		\draw[fill=black](N)circle(2pt)node[below left]{$N$};	
		
		\tkzMarkRightAngle[thick,draw=black](M,X,F1')	
		\tkzMarkRightAngle[thick,draw=black](N,F1,X)	
		\tkzMarkRightAngle[thick,draw=black](N,M,X)	
		
		\tkzMarkSegment[color=black,pos=.5,mark=||](M,N)
		\tkzMarkSegment[color=black,pos=.5,mark=||](F1,X)
		\tkzMarkSegment[color=black,pos=.5,mark=||](F1',X)		
	\end{tikzpicture}
	\caption{Lemma \ref{lem:OX}: the length of segment $OX$ does not depend on the choice of point $M$ on the ellipse.}\label{fig:OX}
\end{figure}
\begin{proof}
The first two statements are a direct application of the minimizing property of the billiard reflection (see Remark \ref{rem:min}), together with focal property (Proposition \ref{prop:focal.property}).

Denote by $N$ the intersection of two lines, as described in the last statement of the lemma.
Then $MNF_1X$ is a rectangle, thus $MN=XF_1$.
Since $X$ is the midpoint of $F_1F_1'$, we have $MN=XF_1'$.
Since segments $MN$ and $XF_1'$ are, in addition to that, parallel, quadrangle $MNXF_1'$ is a parallelogram.
Thus, $XN$ is parallel to the line $F_1'F_2$, so $N\in OX$.
\end{proof}

\begin{lemma}\label{lem:quadratics}
Suppose that $EFGH$ is a parallelogram which represents a periodic billiard trajectory within rectangle $ABCD$ and ellipse with centre $O$ and foci $F_1$, $F_2$, as shown in Figure \ref{fig:analysis}.
Denote the sides of the rectangle by $a=AB=CD$ and $b=BC=AD$, the distances of $F_1$ from $AD$, $AB$ by $x$, $y$ respectively, and set $e=AE$.
Then:
\begin{equation}\label{eq:quadratics}
	(a-2x)^2-(b-2y)^2=a^2-b^2
	\quad\text{and}\quad
	(a-2x)(b-2y)=b(a-2e).
\end{equation}
\begin{figure}[h]
	\begin{center}
		\begin{tikzpicture}[]
		\begin{scope}[yscale=-1,rotate=149.27]
			\coordinate (G) at (-1.75716, 0.985292);
			\coordinate (F) at (-2.31683, -0.45909);
			\coordinate (E) at (1.75716, -0.985292);
			\coordinate (H) at (2.31683, 0.45909);
			
			\coordinate (C) at (-2.80513,0.362314);
			\coordinate (B) at (-1.02193, -2.63736);
			\coordinate (A) at (2.80513,-0.362314);
			\coordinate (D) at (1.02193, 2.63736);
			
			\coordinate (O) at (0,0);
			\coordinate (F1) at (2,0);
			\coordinate (F2) at (-2,0);

			\draw[very thick,gray](0,0) ellipse (2.44949 and 1.41421);

			\draw[thick](E)--(F)--(G)--(H)--cycle;

			\draw[thick](A)--(B)--(C)--(D)--cycle;
			
			\draw[dashed](A)--(C);
			\draw[dashed](B)--(D);
			
					\draw[fill=black](A) circle (2pt) node[below left]{$A$};
						\draw[fill=black](B) circle (2pt) node[below right]{$B$};
						\draw[fill=black](C) circle (2pt) node[above right]{$C$};
						\draw[fill=black](D) circle (2pt) node[above left]{$D$};
						
					\draw[fill=black](E) circle (2pt) node[below]{$E$};
					\draw[fill=black](F) circle (2pt) node[right]{$F$};
					\draw[fill=black](G) circle (2pt) node[above]{$G$};
					\draw[fill=black](H) circle (2pt) node[left]{$H$};

\draw[fill=black](O) circle (2pt) node[below]{$O$};
\draw[fill=black](F1) circle (2pt) node[above]{$F_1$};
\draw[fill=black](F2) circle (2pt) node[below]{$F_2$};

\end{scope}			
		\end{tikzpicture}
		\caption{A convex $4$-periodic trajectory within ellipse and rectangle.}\label{fig:analysis}
	\end{center}
\end{figure}
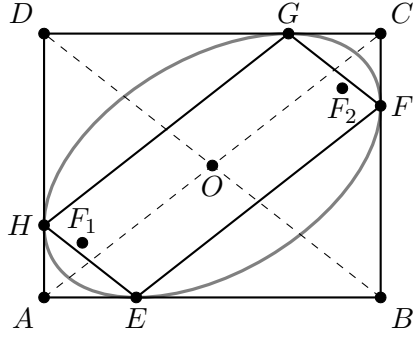
\end{lemma}
\begin{proof}
Denote by $X$, $Y$ and $O_x$, $O_y$ the projections of points $F_1$ and $O$ to $AB$, $AD$ respectively, see the lefthandside of Figure \ref{fig:projF1}.
We notice that $AX=x$, $AY=y$, $AO_x=\dfrac{a}2$, $AO_y=\dfrac{b}2$.
Since Lemma \ref{lem:OX} gives $OX=OY$, the first equality from \eqref{eq:quadratics} is obtained by applying Pythagora's theorem to triangles $OXO_x$ and $OYO_y$.

Now, denote by $Y_1$ the point on $OX$ such that the line $Y_1E$ is orthogonal to $AB$, see the righthandside of Figure \ref{fig:projF1}.
According to the third statement of Lemma \ref{lem:OX}, we have $F_1Y_1\parallel AB$.
Denote $Y_2=F_1Y_1\cap OO_x$.
Now, from the similarity of triangles $OY_1Y_2$ and $OXO_x$, we get the second relation of \eqref{eq:quadratics}.
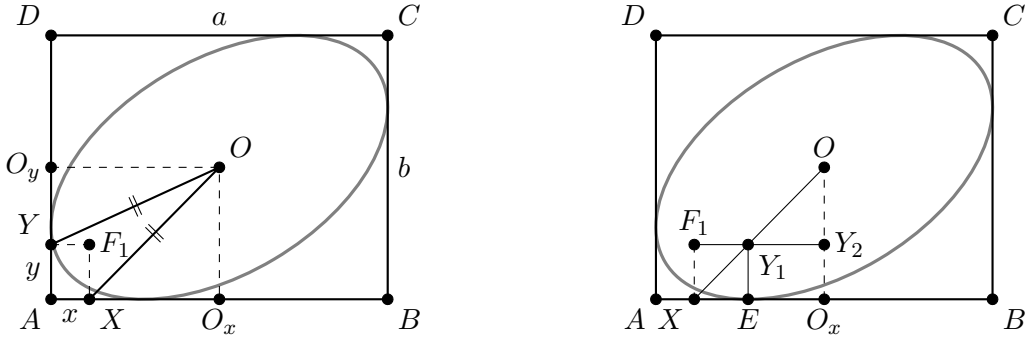
\begin{figure}[h]
	\centering
		\begin{tikzpicture}[]
			\begin{scope}[yscale=-1,rotate=149.27]
				\coordinate (E) at (-1.75716, 0.985292);
				\coordinate (F) at (-2.31683, -0.45909);
				\coordinate (G) at (1.75716, -0.985292);
				\coordinate (H) at (2.31683, 0.45909);
				
				\coordinate (C) at (-2.80513,0.362314);
				\coordinate (B) at (-1.02193, -2.63736);
				\coordinate (A) at (2.80513,-0.362314);
				\coordinate (D) at (1.02193, 2.63736);
				\coordinate (X) at (2.36937, -0.621356);
				\coordinate (Y) at (2.43576, 0.259042);
				
				\coordinate (O) at (0,0);
				\coordinate (F1) at (2,0);
				\coordinate (F2) at (-2,0);
				
				\coordinate (Ox) at (0.8916, -1.49984);
				\coordinate (Oy) at (1.91353, 1.13752);

				\draw[very thick,gray](0,0) ellipse (2.44949 and 1.41421);
				
				 \draw (B) to [edge node={node [right] {$b$}}] (C);
				 \draw (D) to [edge node={node [above] {$a$}}] (C);
				 
				 \draw (A) to [edge node={node [below] {$x$}}] (X);
				 
				 \draw (A) to [edge node={node [left] {$y$}}] (Y);

				\draw[thick](A)--(B)--(C)--(D)--cycle;
				
				\draw[black,dashed](Y)--(F1)--(X);
				\draw[black,dashed](Oy)--(O)--(Ox);
				
				\draw[black,thick](Y)--(O)--(X);

				\draw[fill=black](A) circle (2pt) node[below left]{$A$};
				\draw[fill=black](B) circle (2pt) node[below right]{$B$};
				\draw[fill=black](C) circle (2pt) node[above right]{$C$};
				\draw[fill=black](D) circle (2pt) node[above left]{$D$};

				\draw[fill=black](O) circle (2pt) node[above right]{$O$};
				\draw[fill=black](F1) circle (2pt) node[right]{$F_1$};
				\draw[fill=black](Ox) circle (2pt) node[below]{$O_x$};
				\draw[fill=black](Oy) circle (2pt) node[left]{$O_y$};
				\draw[fill=black](X) circle (2pt) node[below right]{$X$};
				\draw[fill=black](Y) circle (2pt) node[above left]{$Y$};
				
\tkzMarkSegment[color=black,pos=.5,mark=||](O,X)
\tkzMarkSegment[color=black,pos=.5,mark=||](O,Y)
				
				\end{scope}	
\begin{scope}[yscale=-1,shift={(8,0)},rotate=149.27]
	\coordinate (G) at (-1.75716, 0.985292);
	\coordinate (F) at (-2.31683, -0.45909);
	\coordinate (E) at (1.75716, -0.985292);
	\coordinate (H) at (2.31683, 0.45909);
	
	\coordinate (C) at (-2.80513,0.362314);
	\coordinate (B) at (-1.02193, -2.63736);
	\coordinate (A) at (2.80513,-0.362314);
	\coordinate (D) at (1.02193, 2.63736);
	\coordinate (X) at (2.36937, -0.621356);
	\coordinate (Y) at (2.43576, 0.259042);
	\coordinate (Y1) at (1.38779, -0.363942);
	\coordinate (Y2) at (0.522232, -0.878493);
	
	\coordinate (O) at (0,0);
	\coordinate (F1) at (2,0);
	\coordinate (F2) at (-2,0);
	\coordinate (F1') at (2.73874, -1.24271);
	
	\coordinate (Ox) at (0.8916, -1.49984);
	\coordinate (Oy) at (1.91353, 1.13752);

	\draw[very thick,gray](0,0) ellipse (2.44949 and 1.41421);

	\draw[thick](A)--(B)--(C)--(D)--cycle;
	
	\draw[black,dashed](F1)--(X);
	\draw[black,dashed](Ox)--(O);
	\draw[black](Y1)--(E);
%	\draw[black,fill=gray!40,opacity=0.8](F1')--(E)--(Y1)--(X)--cycle;
	\draw[black](X)--(O);
	\draw[black](F1)--(Y2);

	\draw[fill=black](A) circle (2pt) node[below left]{$A$};
	\draw[fill=black](B) circle (2pt) node[below right]{$B$};
	\draw[fill=black](C) circle (2pt) node[above right]{$C$};
	\draw[fill=black](D) circle (2pt) node[above left]{$D$};

	\draw[fill=black](O) circle (2pt) node[above]{$O$};
	\draw[fill=black](F1) circle (2pt) node[above]{$F_1$};
%	\draw[fill=black](F2) circle (2pt) node[above]{$F_2$};
%	\draw[fill=black](F1') circle (2pt) node[below left]{$F_1'$};
	\draw[fill=black](Ox) circle (2pt) node[below]{$O_x$};
	\draw[fill=black](X) circle (2pt) node[below left]{$X$};
	\draw[fill=black](Y1) circle (2pt) node[below right]{$Y_1$};
	\draw[fill=black](Y2) circle (2pt) node[right]{$Y_2$};
	
	\draw[fill=black](E) circle (2pt) node[below]{$E$};
\end{scope}		
		\end{tikzpicture}
		\caption{The proof of Lemma \ref{lem:quadratics}.}\label{fig:projF1}

\end{figure}
\end{proof}

\begin{lemma}\label{lem:F1-exists}
Let $ABCD$ be a rectangle and $E$ a point on the side $AB$.
There is a point $F_1$ in the rectangle, such that it satisfies the relationships \eqref{eq:quadratics}, where $a$, $b$, $e$, $x$, $y$ are defined as in Lemma \ref{lem:quadratics}.

Moreover, point $F_1$ is unique up to the symmetry with respect to the center of the rectangle.
\end{lemma}
\begin{proof}
From \eqref{eq:quadratics}, we get the following biquadratic relationship for $\xi=a-2x$:
$$
\xi^4-(a^2-b^2)\xi^2-b^2(a-2e)^2=0.
$$ 
The discriminant of the corresponding quadratic equation:
\begin{equation}\label{eq:quadeq}
\hat{\xi}^2-(a^2-b^2)\hat\xi-b^2(a-2e)^2=0
\end{equation}
is:
$$
d=(a^2-b^2)^2+4b^2(a-2e)^2.
$$
Notice that always $d\ge0$.

For $d=0$ we have $a=b=2e$, so the rectangle $ABCD$ is a square and we get from \eqref{eq:quadratics} that $F_1$ is its center.

For $d>0$, the equation \eqref{eq:quadeq} has two real solutions $\hat{\xi}_1$, $\hat{\xi}_2$, whose product equals $\hat{\xi}_1\hat{\xi}_2=-b^2(a-2e)^2\le0$.

If $\hat{\xi}_1\hat{\xi}_2=0$, we have $a-2e=0$, i.e.~$E$ is the midpoint of $AB$.
The second relationship of \eqref{eq:quadratics} then gives $a-2x=0$ or $b-2y=0$. 
From the first relationship of \eqref{eq:quadratics}, we can deduce that either $F_1$ is the center of the rectangle if $a=b$, or $x=\dfrac12(a\pm\sqrt{a^2-b^2})$, $y=\dfrac{b}2$ if $a>b$, or $x=\dfrac{a}2$, $y=\dfrac12(b\pm\sqrt{b^2-a^2})$ if $a<b$.

If $\hat{\xi}_1\hat{\xi}_2<0$, then exactly one of the solutions $\hat\xi_1$, $\hat\xi_2$ is positive, so suppose $\hat\xi_1>0$.
Then we have that point $F_1$ satisfies $x=\dfrac12\left(a\pm\sqrt{\hat\xi_1}\right)$, 
$y=\dfrac12\left(b\mp\dfrac{b(a-2e)}{\sqrt{\hat\xi_1}}\right)$.
The choice of signs here determines two points within the rectangle, which are symmetric with respect to its center.
\end{proof}

\begin{proof}[Analytic proof of Theorem \ref{th:parall}]
Let $EFGH$ be a given parallelogram, and $ABCD$ the rectangle determined by external bisectors of the parallelogram.
We assume that $E\in AB$.
According to Lemma \ref{lem:F1-exists}, there is a point $F_1$ in the rectangle which satisfies the relations \eqref{eq:quadratics}.
Denote by $F_2$ the point symmetric to $F_1$ with respect to the center $O$ of the rectangle.
Applying the reasoning from the proofs of Lemmas \ref{lem:OX}-\ref{lem:F1-exists}, the ellipse with foci $F_1$, $F_2$ which contains $E$ will be the unique ellipse
whithin which $EFGH$ represents a closed billiard trajectory.
\end{proof}

Thus, we have concluded the analytic proof of Theorem \ref{th:parall}. As an outcome, we get the exact formulas for the foci of the ellipse.

\section{Nonconvex $4$-periodic elliptical billiard trajectories}\label{sec:Darboux}

\subsection{Characterization of nonconvex $4$-periodic trajectories}\label{sec:4nonconvex}

Following \cite{BLPT}*{Definition 3.4}, we introduce the notion of a Darboux butterfly.

\begin{definition}
Any closed polygonal self-intersecting line $ABCD$, such that $AB\cong CD$ and $AD\cong BC$, is called \emph{a Darboux butterfly} (or \emph{a bow tie}).
See Figure \ref{fig:butterfly-trapezoid}.
\end{definition}
\begin{figure}[h]
	\centering
	\begin{tikzpicture}[]
		\coordinate (G) at (-2, 0);
		\coordinate (K) at (2, 0);
		\coordinate (H) at (1, 2);
		\coordinate (L) at (-1, 2);
		
		\draw[very thick,black](G)--(H)--(K)--(L)--cycle;

\draw[fill=black,black](G) circle (2pt) node[below left]{$A$};
\draw[fill=black,black](H) circle (2pt) node[above right]{$B$};
\draw[fill=black,black](K) circle (2pt) node[below right]{$C$};
\draw[fill=black,black](L) circle (2pt) node[above left]{$D$};

	\end{tikzpicture}
	\caption{A Darboux butterfly.}\label{fig:butterfly-trapezoid}
\end{figure}
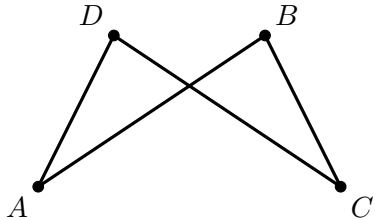

\begin{lemma} 
Every Darboux butterfly consists of legs and diagonals of an isosceles trapezoid.
\end{lemma}
\begin{proof}  
Let $GHKL$ be a Darboux butterfly. Then $\triangle GHK$ is congruent to $\triangle KLG$ and they are placed on the same side of the ``diagonal'' $GK$. Thus, $HL$ is parallel to $GK$, and $GH$ is congruent to $KL$.
\end{proof}

\begin{lemma} 
Every nonconvex $4$-periodic elliptical billiard trajectory has a hyperbola as a caustic and is a Darboux butterfly.
(See Figure \ref{fig:butterfly-ellipse}).
\end{lemma}
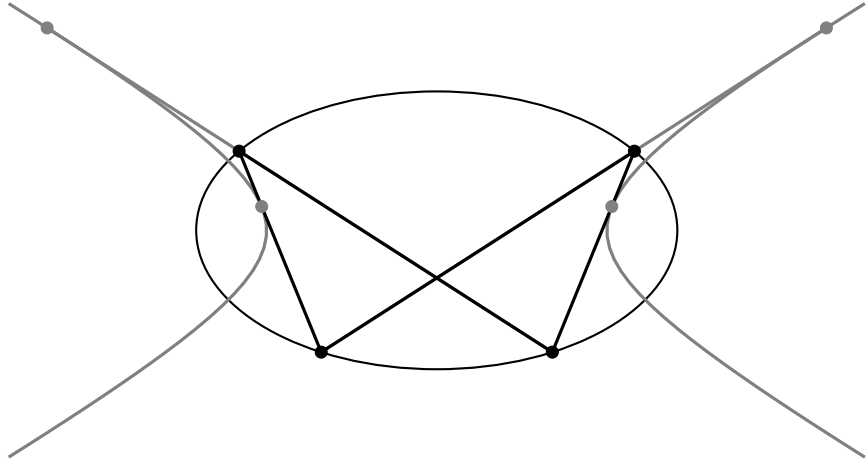
\begin{figure}[h]
	\centering
	\begin{tikzpicture}[scale=1.5]
		\coordinate (G) at (1.74284, 0.698212);
		\coordinate (H) at (1.01905, -1.07417);
		\coordinate (K) at (-1.74284, 0.698212);
		\coordinate (L) at (-1.01905, -1.07417);
		
		\coordinate (D1) at (1.54351,0.21011);
		\coordinate (D2) at (-1.54351,0.21011);
		\coordinate (D3) at (3.43604, 1.78478);
		\coordinate (D4) at (-3.43604, 1.78478);
		
		\draw[thick,black](0,0) ellipse (2.12132 and 1.22474);
		
		\draw [gray, very thick, %dashed,
		  domain=-2:2, samples=40] 
		plot ({1.5*sqrt(1+\x*\x/0.75)}, {\x} );
		
			\draw [gray, very thick,   domain=-2:2, samples=40] 
		plot ({-1.5*sqrt(1+\x*\x/0.75)}, {\x} );

		\draw[very thick,gray](L)--(D3);
\draw[very thick,gray](H)--(D4);	
		
		\draw[very thick,black](G)--(H)--(K)--(L)--cycle;

		\draw[fill=black,black](G) circle (1.5pt);
		\draw[fill=black,black](H) circle (1.5pt);
		\draw[fill=black,black](K) circle (1.5pt);
		\draw[fill=black,black](L) circle (1.5pt);
		\draw[fill=black,gray](D1) circle (1.5pt);
		\draw[fill=black,gray](D2) circle (1.5pt);
		\draw[fill=black,gray](D3) circle (1.5pt);
		\draw[fill=black,gray](D4) circle (1.5pt);
	\end{tikzpicture}
	\caption{A $4$-periodic elliptic billiard trajectory with a hyperbola as the caustic.}\label{fig:butterfly-ellipse}
	
\end{figure}

\begin{proof}
Since any $4$-periodic trajectory with ellipse as caustic is a parallelogram, it follows the caustic of a non-convex $4$-periodic trajectory must be a hyperbola.
Therefore the trajectory is symmetric with respect to the smaller axis of the boundary ellipse, see, for example, \cite{DR2019}*{Theorem 4}. 
Thus, its vertices will be the vertices of an isosceles trapezoid, while its segments are the diagonals and legs of that trapezoid.
\end{proof}	

\begin{remark}
From \cite{St}*{Theorem 1}, it follows that for each
$4$-periodic trajectory of an elliptical billiard, having a hyperbola as the caustic, there exists a $4$-periodic elliptical billiard trajectory having an ellipse as the caustic such that the corresponding sides of the two trajectories are congruent.
\end{remark}

Next, we show that every Darboux butterfly is a billiard trajectory  within a kite.

 \begin{lemma}\label{lem:deltoid} 
 	Every Darboux butterfly is a billiard trajectory within a unique kite, consisting of two congruent acute triangles.
 Moreover, the Darboux butterfly is the union of two $3$-periodic billiard trajectories within each of those congruent acute triangles.
\end{lemma}
\begin{proof}
The edges of the kite are the bisectors of the exterior angles of the Darboux butterfly, see Figure \ref{fig:butterfly-kite}.
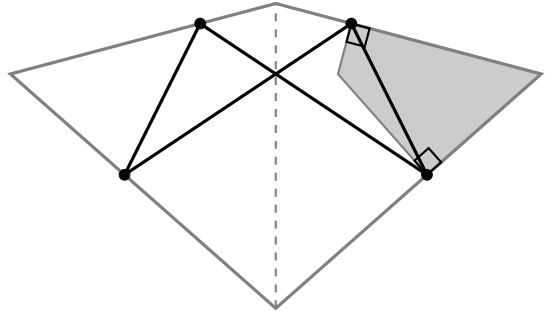
\begin{figure}[h]
	\centering
	\begin{tikzpicture}[]
		\coordinate (G) at (-2, 0);
		\coordinate (K) at (2, 0);
		\coordinate (H) at (1, 2);
		\coordinate (L) at (-1, 2);
		
		\coordinate (A) at (0., -1.76556);
		\coordinate (D) at (-3.51038, 1.33333);
		\coordinate (B) at (3.51038, 1.33333);
		\coordinate (C) at (0,2.26556);
		
		\coordinate (S) at (0.822957, 1.33333);

		\draw[thick,gray,fill=gray!40](S)--(K)--(B)--(H)--cycle;
		
		\draw[very thick,gray](A)--(B)--(C)--(D)--cycle;
		\draw[thick,gray,dashed](A)--(C);
		
		\draw[very thick](G)--(H)--(K)--(L)--cycle;

		\draw[fill=black](G) circle (2pt);
		\draw[fill=black](H) circle (2pt);
		\draw[fill=black](K) circle (2pt);
		\draw[fill=black](L) circle (2pt);
		
		\tkzMarkRightAngle[thick,draw=black](S,K,B)
		\tkzMarkRightAngle[thick,draw=black](S,H,B)

	\end{tikzpicture}
	\caption{The Darboux butterfly is a billiard trajectory within a kite.}\label{fig:butterfly-kite}
\end{figure}
One of the diagonals of the kite, represented by a dashed line in the figure, is the axis of symmetry, and it divides the kite into two congruent triangles.

Consider a shaded quadrangle in Figure \ref{fig:butterfly-kite}, which is determined by two sides of the kite and the bisectors of the corresponding interior angles of the butterfly.
That quadrangle has right angles at the joint vertices with the butterfly, while the angle at the vertex within the butterfly is obtuse.
Thus, the angle of the kite at the vertex which is outside the butterfly must be acute.

Finally, since the kite and the Darboux butterfly are both symmetric with respect to the diagonal of the kite, one can immediately see that the diagonal cuts the Darboux butterfly into two $3$-periodic billiard trajectories within the triangles.
\end{proof}

\begin{lemma}\label{lem:deltoid2} 
For every kite consisting of two congruent acute triangles there is a unique Darboux butterfly which is a $4$-periodic billiard trajectory within the kite. 
This trajectory is obtained as the concatenation $3$-periodic billiard trajectories within each of the two acute triangles.
\end{lemma}
\begin{proof}
Follows from the uniqueness of such triangular billiard trajectories, see Theorem \ref{th:acute}.
\end{proof}

\begin{remark} 
In addition to the unique Darboux butterfly from Lemma \ref{lem:deltoid2}, each kite consisting of two congruent acute triangles has an infinite family of $4$-periodic billiard trajectories.
The segments of such trajectories are parallel to the segments of the Darboux butterfly, see the left-hand side of Figure \ref{fig:4-periodic-kite}.

There is also an infinite family of $4$-periodic billiard trajectories within a kite consisting of two congruent right triangles.
Such trajectories are all isosceles trapezoids, see the right-hand side of Figure \ref{fig:4-periodic-kite}.
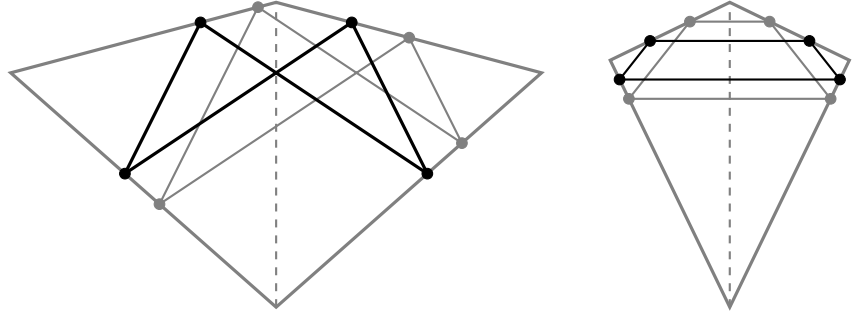
\begin{figure}[h]
	\centering
	\begin{tikzpicture}[]
		\coordinate (G) at (-2, 0);
		\coordinate (K) at (2, 0);
		\coordinate (H) at (1, 2);
		\coordinate (L) at (-1, 2);
		
		\coordinate (A) at (0., -1.76556);
		\coordinate (D) at (-3.51038, 1.33333);
		\coordinate (B) at (3.51038, 1.33333);
		\coordinate (C) at (0,2.26556);
		
		\coordinate(K1) at (2.45727, 0.403663);
		\coordinate(H1) at (1.76001, 1.79817);
		\coordinate(L1) at (-0.239983, 2.20183);
		\coordinate (G1) at (-1.54273, -0.403668);
		
		\coordinate (S) at (0.822957, 1.33333);

		\draw[very thick,gray](A)--(B)--(C)--(D)--cycle;
		
		\draw[thick,gray](H1)--(K1)--(L1)--(G1)--cycle;
		
		\draw[thick,gray,dashed](A)--(C);
		
		\draw[very thick](G)--(H)--(K)--(L)--cycle;

		\draw[fill=black](G) circle (2pt);
		\draw[fill=black](H) circle (2pt);
		\draw[fill=black](K) circle (2pt);
		\draw[fill=black](L) circle (2pt);
		
		\draw[gray,fill=gray](K1) circle (2pt);
		\draw[gray,fill=gray](H1) circle (2pt);
		\draw[gray,fill=gray](L1) circle (2pt);
		\draw[gray,fill=gray](G1) circle (2pt);
		
		\begin{scope}[shift={(6,0)}]
			\coordinate (A1) at (0., -1.76556);
			\coordinate (D1) at (-1.58113, 1.5);
			\coordinate (B1) at (1.58113, 1.5);
			\coordinate (C1) at (0,2.26556);
			
			\coordinate (T1) at (1.05409, 1.75519);
			\coordinate (T2) at (-1.05409, 1.75519);
			\coordinate (T3) at (1.45758, 1.24481);
			\coordinate (T4) at (-1.45758, 1.24481);
			
			\coordinate (P1) at (0.527044, 2.01037);
			\coordinate (P2) at (-0.527044, 2.01037);
			\coordinate (P3) at (1.33402, 0.989627);
			\coordinate (P4) at (-1.33402, 0.989627);
			
			\draw[very thick,gray](A1)--(B1)--(C1)--(D1)--cycle;
			\draw[thick,gray,dashed](A1)--(C1);
			
			\draw[gray,thick](P1)--(P2)--(P4)--(P3)--cycle;
			
			\draw[thick](T1)--(T2)--(T4)--(T3)--cycle;

			\draw[fill=black](T1) circle (2pt);
			\draw[fill=black](T2) circle (2pt);
			\draw[fill=black](T3) circle (2pt);
			\draw[fill=black](T4) circle (2pt);
			
			\draw[gray,fill=gray](P1) circle (2pt);
			\draw[gray,fill=gray](P2) circle (2pt);
			\draw[gray,fill=gray](P3) circle (2pt);
			\draw[gray,fill=gray](P4) circle (2pt);

		\end{scope}

	\end{tikzpicture}
	\caption{$4$-periodic trajectories within kites. The kite on the left consists of two acute triangles, and the one on the right consits of two right triangles.}\label{fig:4-periodic-kite}
	
\end{figure}
\end{remark}

Now, we are ready to prove that each Darboux butterfly is an elliptic billiard trajectory.
	\begin{theorem}\label{th:each-butterfly}
		For each Darboux butterfly, there is a unique ellipse, such that the butterfly is a $4$-periodic billiard trajectory within that ellipse.
	\end{theorem}
	\begin{proof}
		Let $GHKL$ be a given Darboux butterfly.
According to Lemma \ref{lem:deltoid2}, there is a kite $ABCD$, whithin which the butterfly is a periodic billiard trajectory, see Figure \ref{fig:butterfly-construction}.
		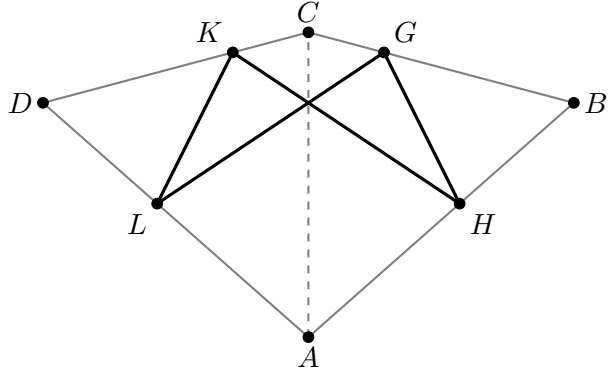
\begin{figure}[h]
			\centering
			\begin{tikzpicture}[]
				\coordinate (L) at (-2, 0);
				\coordinate (H) at (2, 0);
				\coordinate (G) at (1, 2);
				\coordinate (K) at (-1, 2);
				
				\coordinate (A) at (0., -1.76556);
				\coordinate (D) at (-3.51038, 1.33333);
				\coordinate (B) at (3.51038, 1.33333);
				\coordinate (C) at (0,2.26556);
				
				\coordinate (S) at (0.822957, 1.33333);

				\draw[thick,gray,dashed](A)--(C);
				
				\draw[very thick](G)--(H)--(K)--(L)--cycle;
				
				\draw[thick,gray](A)--(B)--(C)--(D)--cycle;

				\draw[fill=black](G) circle (2pt) node[above right]{$G$};
				\draw[fill=black](H) circle (2pt) node [below right]{$H$};
				\draw[fill=black](K) circle (2pt) node[above left]{$K$};
				\draw[fill=black](L) circle (2pt) node[below left]{$L$};
				
				\draw[fill=black](A) circle (2pt) node[below]{$A$};
				\draw[fill=black](B) circle (2pt) node [right]{$B$};
				\draw[fill=black](C) circle (2pt) node[above]{$C$};
				\draw[fill=black](D) circle (2pt) node[left]{$D$};

			\end{tikzpicture}
			\caption{The Darboux butterfly is a billiard trajectory within a kite.}\label{fig:butterfly-construction}
		\end{figure}
Let $\mathcal{C}$ be the unique conic circumscribed about $GHKL$ and touching the line $BC$ at point $G$.	

Applying Pascal's theorem to $\mathcal{C}$ and its points $G$, $G$, $H$, $K$, $K$, $L$, (see Corollary \ref{cor:pascal1}) we get that the intersection points of the following pairs of lines are collinear:
$BC$ and $k$; $GH$ and $KL$; $HK$ and $GL$,
where $k$ is the tangent line to $\mathcal{C}$ at $K$.
Since $GH\cap KL$ and $HK\cap GL$ belong to the line $AC$, we will have that $k\cap BC\in AC$, i.e.~$C\in k$, which gives that $CD$ is tangent to $\mathcal{C}$ at $K$.

This implies that $\mathcal{C}$ is symmetric with respect to the line $AC$ and that the pairs of segments $KL$, $KH$ and $GL$, $GH$ satisfy the billiard reflection law off $\mathcal{C}$ at points $K$, $G$ respectively.
From there, there are conics $\mathcal{C}'$, $\mathcal{C}''$ which are confocal with $\mathcal{C}$, each touching one of those two pairs of segments.

Since $\mathcal{C}$ is symmetric with respect to $AC$, the same must be true for $\mathcal{C}'$ and $\mathcal{C}''$. 
Moreover, since the pairs of segments $KL$, $KH$ and $GL$, $GH$ are also symmetric to each other with respect to $AC$, that will imply $\mathcal{C}'=\mathcal{C}''$.
Therefore, the Darboux butterfly $GHKL$ is inscribed in $\mathcal{C}$ and circumscribed about a conic confocal to $\mathcal{C}$, so it will represent a billiard trajectory within $\mathcal{C}$.

Finally, $\mathcal{C}$ must be an ellipse, since it is touching the sides of the kite $ABCD$ in inner points.
	\end{proof}

\begin{remark}
Notice that, from the uniqueness of the ellipse in Theorem \ref{th:each-butterfly}, it follows that it must be symmetric with respect to the axis of symmetry of the Darboux butterfly.
\end{remark}

\subsection{Analytic approach}\label{sec:Darboux-analytic}

Theorem \ref{th:each-butterfly} was formulated and proved in the classical geometric fashion in Section \ref{sec:4nonconvex}.
Now we provide an analytic proof for that theorem.

\begin{lemma}\label{lem:triangle-focus}
Consider an acute triangle $ABC$. 
Let $G$ and $H$ be the footings of its altitudes from $A$ and $C$ respectively. 
Suppose that $\mathcal E$ is an ellipse whose foci are symmetric with respect to $AC$ such that it is touching $AB$ and $BC$ at $H$ and $G$ respectively. 
Denote by $F_1$ the focus of $\mathcal{E}$ which is placed within $\triangle ABC$, and by $O$, $X$, $Y$, $Z$, $W$ the projections of $F_1$ to $AC$, $AB$, $BC$, $AG$, $CH$ respectively, as shown in Figure \ref{fig:acute-triangle}.

Then, the points $O$, $Z$, $Y$ are collinear as well as the points
$O$, $W$, $X$. 
Moreover, $OX=OY$.
\end{lemma}
\begin{figure}[h]
	\centering
	\begin{tikzpicture}[]
		\coordinate (A) at (0,-2.88675);
		\coordinate (B) at (5,0);
		\coordinate (C) at (0,2.88675);
		\coordinate (G) at (2.5,1.44338);
		\coordinate (H) at (2.5,-1.44338);
		\coordinate (F1) at (2.88675,0);
		\coordinate (F2) at (-2.88675,0);
		\coordinate (O) at (0, 0);
		\coordinate(X) at (3.41506, -0.915064);
		\coordinate(W) at (1.97169, -0.528312);
		\coordinate(Y) at (3.41506, 0.915064);
		\coordinate(Z) at (1.97169, 0.528312);

		\draw[very thick,gray](0,0) ellipse (3.53553 and 2.04124);

		\draw[very thick](A)--(B)--(C)--cycle;
		\draw(A)--(G);
		\draw(C)--(H);
		\draw(O)--(F1)--(X);
		\draw(Z)--(F1)--(Y);
		\draw(F1)--(W);
		\draw[very thick,dotted](Y)--(O)--(X);
		
		\draw[fill=black](A) circle (2pt) node[below]{$A$};
		\draw[fill=black](B) circle (2pt) node [right]{$B$};
		\draw[fill=black](C) circle (2pt) node[above]{$C$};
		\draw[fill=black](G) circle (2pt) node[above]{$G$};
		\draw[fill=black](H) circle (2pt) node[below]{$H$};
		\draw[fill=black](F1) circle (2pt) node[right]{$F_1$};
		\draw[fill=black](O) circle (2pt) node[above left]{$O$};
		\draw[fill=black](X) circle (2pt) node[below]{$X$};	
		\draw[fill=black](Y) circle (2pt) node[above]{$Y$};					
		\draw[fill=black](W)circle(2pt)node[below left]{$W$};			
		\draw[fill=black](Z)circle(2pt)node[above left]{$Z$};			
		
	\end{tikzpicture}
	\caption{Lemma \ref{lem:triangle-focus}.}\label{fig:acute-triangle}
\end{figure}
\begin{proof} 
A direct application of Lemma \ref{lem:OX}.	
\end{proof}

Applying Theorem \ref{th:simson} to triangles $ACG$ and $ACH$, we get the following

\begin{lemma}\label{lem:F}
Let $ABC$ be an arbitrary triangle such that its angles at $A$, $C$ are not right, and $G$, $H$ the footings of its altitudes from $A$, $C$.
Then the locus of all points $F$ in the plane such that its projections to $AC$, $AG$, $BC$ are collinear is the circle with radius $AC$. (See Figure \ref{fig:locusF}).
That circle is also the locus of all points whose projections to $AC$, $CH$, $AB$ are collinear. 
\end{lemma}
\begin{figure}[h]
	\centering
	\begin{tikzpicture}[]
		\coordinate (A) at (0,0);
		\coordinate (B) at (3,4);
		\coordinate (C) at (5,0);
		\coordinate (G) at (4,2);
	\coordinate (H) at (1.8, 2.4);
		\coordinate (F) at (2.77778,2.48452);
		\coordinate (Fp1) at (2.77778,0);
		\coordinate (Fp2) at (3.56175, 2.8765);
		\coordinate(Fp3)at (3.21603, 1.60802);
		\coordinate(Fp4)at(2.19257, 2.92343);
		\coordinate(Fp5)at(2.38521, 1.96109);
		\coordinate (O) at (2.5, 0);
		\coordinate(X) at (3.41506, -0.915064);
		\coordinate(W) at (1.97169, -0.528312);
		\coordinate(Y) at (3.41506, 0.915064);
		\coordinate(Z) at (1.97169, 0.528312);

%		\draw[very thick,gray](0,0) ellipse (3.53553 and 2.04124);
		
		\draw[thick](O)circle(2.5);
		
		\draw[very thick](A)--(B)--(C)--cycle;
		
		\draw[thin,gray](A)--(G);
		\draw[thin,gray](C)--(H);
		\draw[thin,gray](Fp2)--(F)--(Fp1);
		\draw[thin,gray](Fp4)--(F)--(Fp3);
		\draw[thin,gray](F)--(Fp5);
		\draw[very thick,dotted](Fp4)--(Fp1)--(Fp2);
		
		\draw[fill=black](A)circle(2pt)node[left]{$A$};
		\draw[fill=black](B)circle(2pt)node[above]{$B$};
		\draw[fill=black](C)circle(2pt)node[right]{$C$};
		\draw[fill=black](F)circle(2pt)node[above]{$F$};
		\draw[fill=black](H)circle(2pt)node[above left]{$H$};
		\draw[fill=black](G)circle(2pt)node[above right]{$G$};
		
		\draw[fill=black](Fp1)circle(2pt);
		\draw[fill=black](Fp2)circle(2pt);
		\draw[fill=black](Fp3)circle(2pt);
		\draw[fill=black](Fp4)circle(2pt);
		\draw[fill=black](Fp5)circle(2pt);

	\end{tikzpicture}
	\caption{Lemma \ref{lem:F}.}\label{fig:locusF}
\end{figure}

\begin{remark}
It is interesting to note that the locus from Lemma \ref{lem:F} does not depend on the vertex $B$.
\end{remark}

\begin{proposition}\label{prop:unique}
Consider an acute triangle $ABC$.
Let $G$ and $H$ be the footings of its altitudes from $A$ and $C$ respectively. 
Then there is a unique point $F_1$ inside the triangle such that its projections $O$, $X$, $Y$, $Z$, $W$ to $AC$, $AB$, $BC$, $AG$, $CH$ respectively,
satisfy the following:
\begin{itemize}
\item[(i)] $O$, $Z$, $Y$ are collinear;
\item[(ii)] $O$, $W$, $X$ are collinear;
\item[(iii)] $OX=OY$.
\end{itemize}
\end{proposition}
\begin{proof}
Without loss of generality, we may set a coordinate system in the plane such that $A=(0,0)$ and $C=(1,0)$. 
In that system, denote the coordinates of $B$ and $F$ by $(b_1,b_2)$ and $(f_1, f_2)$ respectively.
Applying Lemma \ref{lem:F}, we see that each of the conditions (i) and (ii) is equivalent to $F$ belonging to the circle with radius $AC$, which can be written in the coordinates as:
\begin{equation}\label{eq:CC1'}
	-f_1+f_1^2+f_2^2=0.
\end{equation}
A direct calculation for condition (iii) gives
\begin{equation}\label{eq:cond3}
b_2^2
\left(
(1-2b_1)(f_1^2+f_2^2)+ 2(b_1^2+b_2^2)f_1 -(b_1^2+b_2^2)
\right)=0.
\end{equation}
Observing that $b_2\ne 0$ and substituting $f_2^2=f_1-f_1^2$ from \eqref{eq:CC1'} into \eqref{eq:cond3}, we get
\begin{equation}\label{eq:cond3a}
f_1=\frac{b_1^2+b_2^2}{1-2b_1+2b_1^2+2b_2^2}.
\end{equation}
Thus, $F_1=(f_1, f_2)$ which satisfies the conditions (i), (ii), and (iii) is uniquely determined
through \eqref{eq:cond3a} and \eqref{eq:CC1'}. 
This completes the proof.
\end{proof}

This finishes the analytic proof of Theorem \ref{th:each-butterfly}.

\subsection*{Acknowledgements} 
The authors would like to thank Djordje Barali\'c and Milo\v s Djori\'c for useful discussions. 
The research  was supported
by the Australian Research Council, Discovery Project 190101838 \emph{Billiards within quadrics and beyond}, the Science Fund of Serbia grant \emph{Integrability and Extremal Problems in Mechanics, Geometry and
Combinatorics}, MEGIC, Grant No. 7744592  and the Simons Foundation grant no. 854861.
We thank the anonymous referees for careful reading and useful comments and suggestions. M.~R.~is grateful to Max Planck Institute for Mathematics in Bonn for its hospitality and financial support.

%\appendix


\begin{bibdiv}
\begin{biblist}
\bib{ADSK2016}{article}{
   author={Avila, Artur},
   author={De Simoi, Jacopo},
   author={Kaloshin, Vadim},
   title={An integrable deformation of an ellipse of small eccentricity is
   an ellipse},
   journal={Ann. of Math. (2)},
   volume={184},
   date={2016},
   number={2},
   pages={527--558},
   issn={0003-486X},
   review={\MR{3548532}},
   doi={10.4007/annals.2016.184.2.5},
}
{
\bib{Bar}{article}{
	author={Barali\'c, Dj.},
	title={A short proof of the Bradley theorem},
	journal={American Mathematical Monthly},
	volume={122},
	pages={381--385},
	date={2015}
}
}
\bib{Berg}{article}{
	author={van der Berg, F. J.},
	title={Nogmaals over afgeleide Wortelpunten},
	journal={Nieuw Archiev voor Wiskunde},
	volume={15},
	pages={100--164},
	date={1888}
}

\bib{BergerGeomII}{book}{
	author={Berger, Marcel},
	title={Geometry. II},
	series={Universitext},
	publisher={Springer-Verlag, Berlin},
	date={1987},
	pages={x+406},
}


\bib{BiMi2017}{article}{
   author={Bialy, Misha},
   author={Mironov, Andrey E.},
   title={Angular billiard and algebraic Birkhoff conjecture},
   journal={Adv. Math.},
   volume={313},
   date={2017},
   pages={102--126},
   issn={0001-8708},
   review={\MR{3649222}},
   doi={10.1016/j.aim.2017.04.001},
}

\bib{BiMi2022}{article}{
   author={Bialy, Misha},
   author={Mironov, Andrey E.},
   title={The Birkhoff-Poritsky conjecture for centrally-symmetric billiard
   tables},
   journal={Ann. of Math. (2)},
   volume={196},
   date={2022},
   number={1},
   pages={389--413},
   issn={0003-486X},
   review={\MR{4429262}},
   doi={10.4007/annals.2022.196.1.2},
}


\bib{Bolotin1990}{article}{
    author={Bolotin, S. V.},
    title={Integrable Birkhoff billiards},
    journal={Vestnik Moskov. Univ. Ser. I Mat. Mekh.},
    date={1990},
    number={2},
    pages={33-36}
}


\bib{BLPT}{article}{
	author={Bor, Gil},
	author={Levi, Mark},
 author={Perline, Ron},
	author={Tabachnikov, Sergei},
	title={Tire Tracks and Integrable Curve Evolution },
	journal={Int. Math. Res. Not. IMRN},
	volume={2020},
	date={2020},
	number={9},
	pages={2698--2768},
}

\bib{CZ2021}{article}{
   author={Corvaja, P.},
   author={Zannier, U.},
   title={Finiteness theorems on elliptical billiards and a variant of the Dynamical Mordell-Lang Conjecture},
   journal={Proceedings of the London Mathematical Society},
   volume={127},
   date={2023},
   number={5},
   pages={1268-1337}
}

\bib{Drag2011}{article}{
	author={Dragovi\'{c}, Vladimir},
	title={Poncelet-Darboux curves, their complete decomposition and Marden
		theorem},
	journal={Int. Math. Res. Not. IMRN},
	date={2011},
	number={15},
	pages={3502--3523},
}

\bib{DGR2022}{article}{
	author={Dragovi\'c, Vladimir},
	author={Gasiorek, Sean},
	author={Radnovi\'c, Milena},
	title={Billiard ordered games and books},
	journal={Regul. Chaotic Dyn.},
	volume={27},
	date={2022},
	number={2},
	pages={132--150},
}

\bib{DragRadn2011book}{book}{
	author={Dragovi\'c, Vladimir},
	author={Radnovi\'c, Milena},
	title={Poncelet Porisms and Beyond},
	publisher={Springer Birkhauser},
	date={2011},
	place={Basel}
}

\bib{DragRadn2014jmd}{article}{
    author={Dragovi\'c, Vladimir},
    author={Radnovi\'c, Milena},
    title={Pseudo-integrable billiards and arithmetic dynamics},
    date={2014},
    journal={Journal of Modern Dynamics},
    volume={8},
    number={1},
    pages={109--132}
}

\bib{DR2019}{article}{
	author={Dragovi\'c, Vladimir},
	author={Radnovi\'c, Milena},
	title={Caustics of Poncelet polygons and classical extremal polynomials},
	journal={Regul. Chaotic Dyn.},
	volume={24},
	date={2019},
	number={1},
	pages={1--35},
}

\bib{DR2019cmp}{article}{
	author={Dragovi\'c, Vladimir},
	author={Radnovi\'c, Milena},
	title={Periodic ellipsoidal billiard trajectories and extremal
		polynomials},
	journal={Comm. Math. Phys.},
	volume={372},
	date={2019},
	number={1},
	pages={183--211},
}

\bib{DR2023}{article}{
	author={Dragovi\'c, Vladimir},
	author={Radnovi\'c, Milena},
	title={Billiards within ellipsoids in the 4-dimensional pseudo-Euclidean
		spaces},
	journal={Regul. Chaotic Dyn.},
	volume={28},
	date={2023},
	number={1},
	pages={14--43},
}

\bib{DR2023adv}{article}{
	author={Dragovi\'c, Vladimir},
	author={Radnovi\'c, Milena},
	title={Resonance of ellipsoidal billiard trajectories and extremal
		rational functions},
	journal={Adv. Math.},
	volume={424},
	date={2023},
	pages={Paper No. 109044, 51},
}

\bib{FoVed}{article}{
	author={Fomenko, Anatoly},
	author={Vedyushkina, Viktoria},
	title={Billiards and Intregrable Systems},
	journal={Russian Mathematical Surveys},
	volume={78},
	date={2023},
	number={5},
	pages={881--954},
}

\bib{GKR2021}{article}{
   author={Garcia, Ronaldo},
   author={Reznik, Dan},
   author={Koiller, Jair},
   title={New properties of triangular orbits in elliptic billiards},
   journal={Amer. Math. Monthly},
   volume={128},
   date={2021},
   number={10},
   pages={898--910},
   issn={0002-9890},
   review={\MR{4343349}},
   doi={10.1080/00029890.2021.1982360},
}

\bib{GKR2022}{article}{
   author={Garcia, Ronaldo},
   author={Koiller, Jair},
   author={Reznik, Dan},
   title={Loci of 3-periodics in an elliptic billiard: Why so many
   ellipses?},
   journal={J. Symbolic Comput.},
   volume={114},
   date={2023},
   pages={336--358},
   issn={0747-7171},
   review={\MR{4446194}},
   doi={10.1016/j.jsc.2022.06.001},
}
	

\bib{GSO}{book}{
	author={Glaeser, Georg},
	author={Stachel, Hellmuth},
	author={Odehnal, Boris},
	title={The universe of quadrics},
	publisher={Springer, Berlin},
	date={2020},
}

\bib{Glu2021}{article}{
   author={Glutsyuk, Alexey},
   title={On polynomially integrable Birkhoff billiards on surfaces of
   constant curvature},
   journal={J. Eur. Math. Soc. (JEMS)},
   volume={23},
   date={2021},
   number={3},
   pages={995--1049},
   issn={1435-9855},
   review={\MR{4210728}},
   doi={10.4171/jems/1027},
}


\bib{Kal2008}{article}{
	author={Kalman, D.},
	title={An Elementary Proof of Marden's Theorem},
	journal={The American Mathematical Monthly},
	volume={115},
	date={2008},
	pages={330--337}
}

\bib{KS2018}{article}{
   author={Kaloshin, Vadim},
   author={Sorrentino, Alfonso},
   title={On the local Birkhoff conjecture for convex billiards},
   journal={Ann. of Math. (2)},
   volume={188},
   date={2018},
   number={1},
   pages={315--380},
   issn={0003-486X},
   review={\MR{3815464}},
   doi={10.4007/annals.2018.188.1.6},
}


\bib{KozTrBIL}{book}{
	author={Kozlov, Valery},
	author={Treshch\"ev, Dmitry},
	title={Billiards},
	publisher={Amer. Math. Soc.},
	address={Providence RI},
	date={1991}
}

\bib{Kozlov2003}{article}{
    author={Kozlov, V. V.},
    title={Rationality conditions for the ratio of elliptic
            integrals and the great Poncelet theorem},
    language={Russian},
    journal={Vestnik Moskov. Univ. Ser. I Mat. Mekh.},
    volume={71},
    date={2003},
    pages={6--13},
    number={4},
    translation={
        language={English},
        journal={Moscow Univ. Math. Bull.},
        volume={58},
        number={4},
        date={2004},
        pages={1--7}
    }
}


\bib{MardenPOLY}{book}{
	author={Marden, Morris},
	title={Geometry of Polynomials},
	publisher={AMS},
	series={Math. Surveys},
	volume={2},
	edition={2},
	date={1966}
}

\bib{Pra}{book}{
	author={Prasolov, Victor V.},
	title={Polynomials},
	series={Algorithms and Computation in Mathematics},
	volume={11},
	note={Translated from the 2001 Russian second edition by Dimitry Leites},
	publisher={Springer-Verlag, Berlin},
	date={2004},
}

\bib{Scho}{book}{
	author={Schoenberg, Isaac J.},
	title={Mathematical time exposures},
	publisher={Mathematical Association of America, Washington, DC},
	date={1982},
}

\bib{Sch1}{article}{
	author={Schwartz, Richard Evan},
	title={Obtuse triangular billiards. I. Near the $(2,3,6)$ triangle},
	journal={Experiment. Math.},
	volume={15},
	date={2006},
	number={2},
	pages={161--182},
}

\bib{Sch2}{article}{
	author={Schwartz, Richard Evan},
	title={Obtuse triangular billiards. II. One hundred degrees worth of
		periodic trajectories},
	journal={Experiment. Math.},
	volume={18},
	date={2009},
	number={2},
	pages={137--171},
}

\bib{Si1864}{article}{
	author={Siebeck, J.},
	title={Ueber eine neue analytische Behandlungweise der Brennpunkte},
	journal={J. Reine Angew. Math.},
	volume={64},
	date={1864},
	pages={175--182}
}

\bib{St}{article}{
	author={Stachel, Hellmuth},
	title={Isometric Billiards in Ellipses and Focal
Billiards in Ellipsoids
},
	journal={Journal for Geometry and Graphics},
	date={2021},
 volume={25},
	number={1},
	pages={97--118},
}

\bib{Tab2005book}{book}{
	author={Tabachnikov, Serge},
	title={Geometry and billiards},
	series={Student Mathematical Library},
	volume={30},
	publisher={American Mathematical Society},
	place={Providence, RI},
	date={2005}
}

\bib{Tab2022}{article}{
	author={Tabachnikov, Serge},
	title={Remarks on rigidity properties of conics},
	journal={Regul. Chaotic Dyn.},
	volume={27},
	date={2022},
	number={1},
	pages={18--23},
}
\end{biblist}
\end{bibdiv}
\end{document}